\documentclass[12pt]{amsart}

\usepackage{tikz,pgf,amscd,enumerate,multicol}
\usepackage{amsmath,amssymb,tikz-cd,tkz-graph}%

\usepackage[pagebackref,colorlinks,linkcolor=red,citecolor=blue,urlcolor=blue,hypertexnames=false]{hyperref}

\usepackage[margin=1.2cm]{geometry}
\usepackage{float}
\usetikzlibrary{decorations.markings,arrows}

\tikzset{midarrow/.style = {postaction=decorate, decoration={markings,mark=at position .5 with \arrow{stealth}}}}

\pgfdeclarelayer{background}
\pgfsetlayers{background,main}

\usepackage[all,2cell]{xy}

\newcommand{\NN}{\mathbb{N}}

\newcommand{\ZZ}{\mathbb{Z}}

\newcommand{\SC}{\operatorname{Soc}}
\newcommand{\supp}{\operatorname{supp}}

\newcommand{\MCE}{\operatorname{MCE}}
\newcommand{\KP}{\operatorname{KP}}
\newcommand{\curly}{\mathrel{\leadsto}}

\newtheorem{thm}{Theorem}[section]
\newtheorem{cor}[thm]{Corollary}
\newtheorem{lem}[thm]{Lemma}
\newtheorem{prop}[thm]{Proposition}

\newtheorem{prob}[thm]{Program}

\theoremstyle{definition}
\newtheorem{dfn}[thm]{Definition}
\newtheorem{dfns}[thm]{Definitions}

\theoremstyle{remark}
\newtheorem{rmk}[thm]{Remark}
\newtheorem{rmks}[thm]{Remarks}
\newtheorem{example}[thm]{Example}
\newtheorem{examples}[thm]{Examples}

\newcommand{\lspan}{\operatorname{span}}

\def\mG{\mathcal{G}}
\def\G{\Gamma}
\newcommand{\Typ}{\operatorname{Typ}}
\newcommand{\typ}{\operatorname{typ}}
\newcommand{\gr}{\operatorname{gr}}
\newcommand{\gb}{\operatorname{gb}}

\begin{document}

\title[The Talented Monoid of Higher-Rank Graphs]{The Talented Monoid of Higher-Rank Graphs\\ with Applications to Kumjian-Pask Algebras}

\author[R. Hazrat]{Roozbeh Hazrat}
\address{(R. Hazrat) Centre for Research in Mathematics and Data Science, Western Sydney University, Sydney NSW 2150, Australia}
\email{r.hazrat@westernsydeny.edu.au}

\author[P. Mukherjee]{Promit Mukherjee}
\address{(P. Mukherjee) Department of Mathematics, Jadavpur University, Kolkata-700032, India}
\email{promitmukherjeejumath@gmail.com}

\author[D. Pask]{David Pask}
\address{(D. Pask) School of Mathematics and Applied Statistics, University of Wollongong, Wollongong NSW 2522, Australia}
\email{david.a.pask@gmail.com}

\author[S.K. Sardar]{Sujit Kumar Sardar}
\address{ (S. K. Sardar) Department of Mathematics, Jadavpur University, Kolkata-700032, India}
\email{sksardarjumath@gmail.com} 

\subjclass[2020]{16W50, 16S99, 16D25, 18B40}

\keywords{Higher-rank graph, graph monoid, talented monoid, Kumjian-Pask algebra, aperidocity, cofinality, type semigroup, semisimplicity}

\begin{abstract}
Given a row-finite higher-rank $k$-graph $\Lambda$, we define a commutative monoid $T_\Lambda$ which is a higher-rank analogue of the talented monoid of a directed graph introduced in (J. Algebra \textbf{384} (2013), 242-266) and (J. Algebra \textbf{547} (2020), 430-455). The \emph{talented monoid} $T_\Lambda$ is canonically a $\mathbb{Z}^k$-monoid with respect to the action of state shift. This monoid coincides with the positive cone of the graded Grothendieck group $K_0^{\gr}(\KP_\mathsf{k}(\Lambda))$ of the Kumjian-Pask algebra $\KP_\mathsf{k}(\Lambda)$ with coefficients in a field $\mathsf{k}$. The aim of the paper is to investigate this $\mathbb{Z}^k$-monoid as a capable invariant for classification of Kumjian-Pask algebras. 

We establish the $\mathbb{Z}^k$-monoid isomorphism $T_\Lambda\cong T_{\mathcal{G}_\Lambda}=\Typ(\mathcal{G}_\Lambda\times_c \mathbb{Z}^k)$, the type monoid of the inverse semigroup of the compact open bisections of the skew-product groupoid $\mathcal{G}_\Lambda\times_c \mathbb{Z}^k$ where $c:\mathcal{G}_\Lambda\longrightarrow \mathbb{Z}^k$ is the canonical cocycle. From this it follows that $T_\Lambda$ is a refinement monoid. If $\mathbb{Z}^k$ acts freely on $T_\Lambda$ (i.e., if $T_\Lambda$ has no nonzero periodic element), then we show that the $k$-graph $\Lambda$ is aperiodic in the sense of Lewin and Sims (Math. Proc. Camb. Phil. Soc. \textbf{149} (2010), 333-350). The converse is also proved to be true provided $\Lambda$ has no sources and $T_\Lambda$ is atomic. Moreover in this case, we provide a talented monoid characterization for strongly aperiodic $k$-graphs. We prove that for a row-finite $k$-graph $\Lambda$ without sources, cofinality is equivalent to the simplicity of $T_\Lambda$ as a $\mathbb{Z}^k$-monoid. In view of this we provide a talented monoid criterion for the Kumjian-Pask algebra $\KP_R(\Lambda)$ of $\Lambda$ over a unital commutative ring $R$ to be graded basic ideal simple. We also describe the minimal left ideals of $\KP_\mathsf{k}(\Lambda)$  in terms of the aperiodic atoms of $T_\Lambda$ and thus obtain a monoid theoretic characterization for $\SC(\KP_\mathsf{k}(\Lambda)$) to be an essential ideal. These results help us to characterize semisimple Kumjian-Pask algebras through the lens of $T_\Lambda$. 
\end{abstract}

\maketitle

\section{Introduction}\label{sec introduction}
Let $\mG$ be an ample Hausdorff groupoid, $\G$ a discrete group and 
$c\colon \mG \to \Gamma$ a cocycle. Consider the skew-product groupoid $\mG \times_c \Gamma$. The homogeneous compact open bisections of this groupoid form a Boolean inverse semigroup \cite{Bosa, Cordeiro}. We call the type monoid of this inverse semigroup, the \emph{talented monoid} of $\mG$:
\[T_\mG := \Typ(\mG \times_c \Gamma)\] 
which is a commutative $\G$-monoid generated by symbols $\typ(U \times \alpha)$, where $U$ is a compact open subset of $\mG^{0}$ and $\alpha \in \Gamma$, with an action of $\Gamma$ on $T_\mG$ defined 
 on generators by 
\[{}^\beta \typ(U \times \alpha) := \typ(U \times \beta \alpha)\]
with $\beta \in \G$ and extended linearly to all elements. 

This paper pursue the following program for the case of higher-rank graphs: 

\begin{prob}\label{genproblem}
Describe the class of $\Gamma$-graded ample groupoids $\mG$ such that the talented monoid $T_\mG$, as a $\Gamma$-monoid, is a complete invariant for Steinberg and groupoid $C^*$-algebras. 
\end{prob}

When $E$ is a directed graph and $\mG_E$ is the associated graph groupoid, then $T_{\mG_E}$ is the talented monoid of the graph $E$ \cite{Hazrat-main, Hazrat} which coincides with the positive cone of the graded Grothendieck group $K_0^{\gr}(L_\mathsf{k}(E))$  of Leavitt path algebra $L_\mathsf{k}(E)$ with  coefficient in a field $\mathsf{k}$. ~\cite{Hazrat}. Then Program~\ref{genproblem} takes the form of the Graded Classification Conjecture~\cite{willie,mathann}. Till date a lot of evidence has been found in the support of this conjecture (see \cite{Hazrat,Cordeiro}), also it has been proved to be true for some classes of directed graphs; for details see \cite{willie}.

The notion of higher-rank graphs (or, $k$-graphs) was first introduced in \cite{Kumjian-Pask} to obtain a graph theoretic framework of the study performed by Robertson and Steger in the seminal papers \cite{RS1, RS2}. Modelled on the path category of a directed graph, higher-rank graphs, their $C^*$-algebras (see \cite{Kumjian-Pask,Yeend,Robertson} etc.) and their algebraic counterparts namely Kumjian-Pask algebras (see \cite{Pino,Flynn,Clark}) cover a wider class of examples including graph $C^*$-algebras \cite{KPR} and Leavitt path algebras \cite{Abrams,Abrams-Monograph} respectively.

Given a row-finite $k$-graph with no sources, Kumjian and the third named author in \cite{Kumjian-Pask} associated an ample Hausdorff groupoid $\mathcal{G}_\Lambda$ which is known as the \emph{path groupoid} of $\Lambda$. This gives a groupoid model for the higher-rank graph $C^*$-algebra as it was shown in \cite[Corollary 3.5]{Kumjian-Pask} that $C^*(\Lambda)\cong C^*(\mathcal{G}_\Lambda)$. On the other hand it follows from the work of Clark and Pangalela \cite[Proposition 5.4]{Clark} that the Steinberg algebra of the groupoid $\mathcal{G}_\Lambda$ is isomorphic to the Kumjian-Pask algebra $\KP_R(\Lambda)$ where $R$ is a unital commutative ring. As a consequence of this, fundamental algebraic properties like simplicity and basic ideal simplicity of the Kumjian-Pask algebra are reflected by certain properties (minimality and effectiveness) of the path groupoid $\mathcal{G}_\Lambda$. Thus to translate Program \ref{genproblem} in the language of higher-rank graphs, we should focus on the talented monoid $T_{\mathcal{G}_\Lambda}$ and observe how the algebraic properties of the higher-rank graph algebras are governed by $T_{\mathcal{G}_\Lambda}$; this is what we have tried to accomplish in the present paper.

As the talented monoid of a directed graph is the graded version of the graph monoid (see \cite{Ara,APW}), we first define \emph{$k$-graph monoids} as higher-rank graph analogue of graph monoids. However the transition from directed graphs to $k$-graphs is not straightforward at all since in a general $k$-graph where $k\ge 2$, a vertex can receive paths from several different directions while in a graph a vertex can emit edges in only one direction. Keeping this fact in mind, we suitably define a monoid $M_\Lambda$ for a row-finite $k$-graph $\Lambda$. One of the important properties of graph monoids which transits also to the talented monoid (via the natural isomorphism of talented monoid with the graph monoid of the covering graph) is the \emph{Confluence property}. We observe that unlike graph monoid, $k$-graph monoid $M_\Lambda$ of a $k$-graph $\Lambda$ with sources may not have the confluence property; see Remark \ref{rem confluence may not follow}. Also we exhibit a $k$-graph monoid $M_\Lambda$ (see Example \ref{ex k-graph monoid which is not a graph monoid}) which can not be realized as the graph monoid $M_E$ of any row-finite directed graph $E$. Looking at these obstacles which arise for $k$-graphs with sources, we decide to concentrate only on row-finite $k$-graphs which do not have any source and prove that the monoids of such $k$-graphs indeed admit a suitable confluence property (see Lemma \ref{lem the confluence lemma}). 

We define the \emph{talented monoid} for a row-finite $k$-graph $\Lambda$ as the commutative monoid generated by $\Lambda^0\times \mathbb{Z}^k$ subject to the relations:
\begin{center}
    $v(n)=\displaystyle{\sum_{\alpha\in v\Lambda^m}} s(\alpha)(n+m)$ for each $v\in \Lambda^0$ with $v\Lambda^m\neq \emptyset$.
\end{center}
There is a natural action of $\mathbb{Z}^k$ on $T_\Lambda$ given by state shift: $^n v(m):=v(m+n)$ for all $v\in \Lambda^0$ and $n,m\in \mathbb{Z}^k$. This makes $T_\Lambda$ a $\mathbb{Z}^k$-monoid which is isomorphic to the $k$-graph monoid $M_{\overline{\Lambda}}$ (see Theorem \ref{thm:talisrefine}) where $\overline{\Lambda}$ is the canonical skew-product of $\Lambda$ by $\mathbb{Z}^k$. Interestingly enough, the monoid $T_\Lambda$ has two more realizations. Firstly, in Proposition \ref{pro realizing as a direct limit}, we show that $T_\Lambda$ is the direct limit of a certain directed system considered by Ara et al. \cite{Ara-Hazrat} in the study of the graded $K$-groups of Kumjian-Pask algebras. As a consequence, $T_\Lambda$ becomes isomorphic with the graded $\mathcal{V}$-monoid of the Kumjian-Pask algebra of $\Lambda$ over a field $\mathsf{k}$ and its group completion becomes the graded Grothendieck group $K_0^{\gr}(\KP_\mathsf{k}(\Lambda))$. Secondly, in Theorem \ref{th isomorphism between talented monoid and type monoid}, we prove that $T_\Lambda$ is isomorphic (as $\mathbb{Z}^k$-monoid) to the talented monoid $T_{\mathcal{G}_\Lambda}$ of the path groupoid $\mathcal{G}_\Lambda$. It then readily follows that $T_\Lambda$ is a conical refinement monoid. Moreover, in view of this result, the isomorphism $T_E\cong T_{\mathcal{G}_E}$ obtained in \cite{Cordeiro} in the level of directed graphs, is perfectly generalized in the setting of $k$-graphs. 

In the attempt of gathering evidence that $T_\Lambda$ is a possible complete invariant for most of the Kumjian-Pask algebras, we initiate in this paper, the study of connecting the monoid theoretic properties of $T_\Lambda$ as well as the properties based on the natural $\mathbb{Z}^k$-action with the geometry of the underlying $k$-graph $\Lambda$. In this direction, here we mainly characterize the two most important properties of $\Lambda$ namely \emph{aperiodicity} and \emph{cofinality}. 

Aperiodicity, in the setting of $k$-graphs, is the suitable analogue of Condition $(L)$ of directed graphs which states that every cycle in the directed graph has an exit. In the literature there are various formulations of aperiodicity depending on the complexity of the $k$-graph concerned e.g., see \cite{Kumjian-Pask, Robertson, Yeend, Lewin}. However these are all equivalent when considered for row-finite $k$-graphs with no sources. Aperiodicity is the main hypothesis in the Cuntz-Krieger uniqueness theorem for $C^*$-algebras (see \cite[Theorem 4.6]{Kumjian-Pask}) and Kumjian-Pask algebras (see \cite[Theorem 4.7]{Pino}). We prove in Proposition \ref{pro sufficient condition of aperiodicity}, that if the group $\mathbb{Z}^k$ acts freely on the monoid $T_\Lambda$ then $\Lambda$ is aperiodic in the sense of Lewin and Sims (see Definition \ref{def LS convention}). However unlike directed graphs, the converse does not hold for general $k$-graphs; see Examples \ref{ex counter example for aperiodicity}, \ref{ex simple but not semisimple}. By establishing that $T_\Lambda$ satisfies the confluence property, we show that if $T_\Lambda$ is \emph{atomic} then it is also a \emph{unique factorization monoid}. In this case the converse is proved to be true (see Theorem \ref{th Equivalence for aperiodicity in UF}). Furthermore, for a row-finite $k$-graph without sources $\Lambda$ with atomic talented monoid, we provide a necessary and sufficient condition for \emph{strong aperiodicity} of $\Lambda$, which is the $k$-graph theoretic generalization of Condition $(K)$ used in the study of Leavitt path algebras and graph $C^*$-algebras. 

Cofinality on the other hand, is the property of a $k$-graph which is a necessary and sufficient requirement for the graded ideal simplicity of the associated Kumjian-Pask algebra. Indeed, in \cite[Lemma 5.12]{Pino}, it was proved that if $\Lambda$ is cofinal then the only hereditary and saturated subsets of $\Lambda^0$ are the trivial ones. By establishing a lattice isomorphism in Proposition \ref{pro the lattice isomorphism} between the lattice of hereditary saturated subsets of $\Lambda^0$ and the lattice of $\mathbb{Z}^k$-order ideals of $T_\Lambda$, here we show that the converse implication also holds. Also, this helps us to show that the simplicity of $T_\Lambda$ as a $\mathbb{Z}^k$-monoid, is equivalent to the cofinality of $\Lambda$ (see Proposition \ref{pro sufficient condition for cofinality}). Therefore our results show that the properties of Kumjian-Pask algebras of being graded ideal simple, graded basic ideal simple and ideal simple remain invariant under $\mathbb{Z}^k$-isomorphism of the talented monoids. 

In \cite{Brown} the authors identifies the socle of the Kumjian-Pask algebra $\KP_\mathsf{k}(\Lambda)$ by introducing \emph{line points} in the $k$-graph $\Lambda$. They obtained several equivalent conditions for semisimplicity of $\KP_\mathsf{k}(\Lambda)$. In Theorem \ref{th line points characrerization}, we show that line points of a $k$-graph $\Lambda$ can be described completely in terms of the $\mathbb{Z}^k$-orbits of the aperiodic atoms of $T_\Lambda$. It provides us with a talented monoid realization of the socle of $\KP_\mathsf{k}(\Lambda)$. Based on this we characterize semisimple Kumjian-Pask algebras through the lens of $T_\Lambda$ and end up adding some more conditions to the list given in \cite[Theorem 4.3]{Brown}. Our result shows that for a row-finite $k$-graph $\Lambda$ which do not have sources, $\KP_\mathsf{k}(\Lambda)$ is semisimple if and only if $T_\Lambda$ is atomic and $\mathbb{Z}^k$ acts freely on $T_\Lambda$. As both these properties are $\mathbb{Z}^k$-isomorphism invariant, we conclude that for two row-finite $k$-graphs without sources $\Lambda,\Omega$ whose talented monoids are $\mathbb{Z}^k$-isomorphic, $\KP_\mathsf{k}(\Lambda)$ is semisimple if and only if $\KP_\mathsf{k}(\Omega)$ is semisimple. Moreover, considering the matrix decomposition of these algebras we are able to show that in this case the Kumjian-Pask algebras $\KP_\mathsf{k}(\Lambda)$ and $\KP_\mathsf{k}(\Omega)$ are isomorphic as $\mathsf{k}$-algebras (see Theorem \ref{th isomorphic TM gives isomorphic KPA}). 

We briefly describe the contents of the paper. In order to make the article self-contained, in Section \ref{sec preliminaries} we include a bit lengthy list of preliminary concepts related to higher-rank graphs, commutative monoids, groupoids and their type semigroups. The main results are given in Sections \ref{sec the talented monoid}--\ref{sec aperiodicity and KP algebra}. We define the $k$-graph monoid $M_\Lambda$ in Section \ref{sec the talented monoid}, prove a confluence lemma for these monoids when $\Lambda$ is a row-finite $k$-graph without sources and then introduce the main object of study, the talented monoid $T_\Lambda$. In section \ref{sec cofinality}, we show that the lattices $\mathcal{H}(\Lambda)$ (lattice of hereditary saturated subsets of $\Lambda^0$), $\mathcal{L}(T_\Lambda)$ (lattice of $\mathbb{Z}^k$-order ideals of $T_\Lambda$) and $\mathcal{L}^{\gb}(\KP_R(\Lambda))$ (lattice of all graded basic ideals of $\KP_R(\Lambda)$) are mutually lattice isomorphic. This helps us to connect cofinality of $\Lambda$ with the property of $T_\Lambda$. In Section \ref{sec aperiodicity and KP algebra} amongst other results, we characterize line points of $\Lambda$ via the talented monoid and consequently prove that $\Lambda$ is aperiodic as a $k$-graph if and only if every atom of $T_\Lambda$ is aperiodic if and only if $\mathbb{Z}^k$ acts freely on $T_\Lambda$. The line point characterization also gives us the ingredient to obtain a talented monoid criteria  for $\SC(\KP_\mathsf{k}(\Lambda))$ to be an essential ideal and for the semisimplicity of $\KP_\mathsf{k}(\Lambda)$. We conclude in Section \ref{sec examples} by giving some examples which illustrate the results obtained throughout the paper.

\section{Preliminaries}\label{sec preliminaries}
In this section we include the background information needed in the development of the paper. Readers who are well familiar with all these concepts may directly proceed to Section \ref{sec the talented monoid}.
\subsection{Higher-rank graphs}
Throughout this paper for any positive integer $k$, $\mathbb{N}^k$ stands for the set of all $k$-tuples of non-negative integers. This can be realized simultaneously as an commutative monoid under coordinatewise addition with identity $0$ and as a small category with only one object. For two $k$-tuples $m=(m_1,m_2,...,m_k)$ and $n=(n_1,n_2,...,n_k)$ in $\mathbb{N}^k$, we write $m\le n$ if $m_i\le n_i$ for all $i=1,2,...,k$. This defines a partial order on $\mathbb{N}^k$ which in turn makes it into a lattice with join and meet defined as the coordinatewise maximum and minimum respectively. For $i=1,2,...,k$ we denote $e_i$ to be the $k$-tuple with $1$ at the $i^{\text{th}}$ coordinate and $0$ at the other coordinates.
\begin{dfn}\label{def higher-rank graph}
A $k$-graph is defined as a $5$-tuple $\Lambda=(\Lambda^0,\Lambda,r,s,d)$ where $\Lambda$ is a countable small category with (countable) set of objects $\Lambda^0$ (which can also be viewed as a subset of $\Lambda$ by identifying each object with the identity morphism on it), $r,s:\Lambda\longrightarrow \Lambda^0$ are the \textit{range} and \textit{source} maps, and $d:\Lambda\longrightarrow \mathbb{N}^k$ is a functor (called the \textit{degree} functor) which satisfies the \emph{unique factorization property}: for every $\lambda\in \Lambda$ with $d(\lambda)=m+n$, there exist \emph{unique} $\mu_1,\mu_2\in \Lambda$ such that $\lambda=\mu_1\mu_2$, $d(\mu_1)=m$, $d(\mu_2)=n$.
\end{dfn}

The factorization property is equivalent to saying that the canonical map 
\begin{align*}
\{(\lambda,\mu)~|~d(\lambda)=m,d(\mu)=n, s(\lambda)=r(\mu)\} & \longrightarrow d^{-1}(m+n)\\ 
(\lambda,\mu) & \longmapsto \lambda\mu 
\end{align*}
is bijective. We often call the arrows or morphisms of the category $\Lambda$ as \emph{paths} and the objects as \emph{vertices}.

We fix some notations related with $k$-graphs which will be used frequently in the paper. For each $n\in \mathbb{N}^k$, we write $\Lambda^n:=d^{-1}(n)$. For objects $u, v\in \Lambda^0$, $v\Lambda:=r^{-1}(v)$, $\Lambda u:=s^{-1}(u)$ and $v\Lambda u:=r^{-1}(v)\cap s^{-1}(u)$. Similarly the sets $v\Lambda^n$, $\Lambda^n u$ and $v\Lambda^n u$ are respectively defined as $r^{-1}(v)\cap d^{-1}(n)$, $s^{-1}(u)\cap d^{-1}(n)$ and $r^{-1}(v)\cap s^{-1}(u)\cap d^{-1}(n)$. For $U,V\subseteq \Lambda$, we write $UV:=\{\lambda\mu~|~\lambda\in U, \mu\in V$ and $s(\lambda)=r(\mu)\}$. 

Let us look at some examples. 

\begin{examples} \label{ex higher-rank graphs}
\begin{enumerate}[$(i)$]

\item A $1$-graph is simply the path category of a directed graph (see \cite{Kumjian-Pask}).

\item Let $\Lambda_k = \NN^k$ where $d$ is the identity map. This is a $k$-graph with only one object.

\item (\textbf{The grid graph}) Let $\Omega_k = \{ (m,n) \in \mathbb{N}^k \times \mathbb{N}^k : m \le n \}$, and $\Omega_k^0 = \mathbb{N}^k$. Define $r,s :\Omega_k \longrightarrow
\Omega_k^0$ by $r (m,n) = m$,  $s (m,n) = n$,
and for $m \le n \le p \in \NN^k$ define $(m,n)(n,p)=(m,p)$ and $d (m,n) = n-m$. Then $( \Omega_k , d )$ is a $k$-graph where $\Omega_k^0$ is identified with $\{ (m,m) : m \in \mathbb{N}^k \}\subseteq \Omega_k$. This $k$-graph is called the \emph{infinite grid graph} as its $1$-\emph{skeleton} (see \cite{Raeburn} for more details) looks like an infinite grid in $k$-dimensional space. For each $m\in \mathbb{N}^k$, one may define a finite grid graph $\Omega_{k,m}$ just by restricting the set of objects as $\Omega_{k,m}^0:=\{n\in \mathbb{N}^k~|~n\le m\}$ and keeping the other things as it is.

\item We may similarly define a $k$-graph $\Delta_k$ with $\Delta_k = \{ (m,n) \in \mathbb{Z}^k \times \mathbb{Z}^k : m \le n \}$, and $\Delta_k^0 = \mathbb{Z}^k$.

\item (\textbf{$2$-graph from $1$-graph}) Suppose $\Lambda$ is a $k$-graph and $f:\mathbb{N}^l\longrightarrow \mathbb{N}^k$ is any monoid morphism. Then one can construct an $l$-graph $f^*(\Lambda)$ (see \cite[Definition 1.9]{Kumjian-Pask}) where 
\begin{displaymath}
f^*(\Lambda):=\{(\lambda,n)\in \Lambda\times \mathbb{N}^l~|~d(\lambda)=f(n)\}
\end{displaymath}
with $s((\lambda,n)):=(s(\lambda),0)$, $r((\lambda,n)):=(r(\lambda),0)$ and $d((\lambda,n)):=n$. As a particular case if $E$ is a directed graph and $f:\mathbb{N}^2\longrightarrow \mathbb{N}$ is defined by $f((x,y)):=x+y$ then we may form the $2$-graph $f^*(E^*)$. For example if $R_1$ is the \textit{rose with $1$-petal} then $f^*(R_1^*)$ is the $2$-graph whose $1$-skeleton consists of a single vertex with two loops on it; one of degree $(1,0)$ and the other one is of degree $(0,1)$. Note that it coincides with the $2$-graph $\Lambda_2$ of $(ii)$. 
\end{enumerate}
\end{examples}

A $k$-graph $\Lambda$ is said to be \textit{row-finite} if $|v\Lambda^n|< \infty$ for all $v\in \Lambda^0$ and $n\in \mathbb{N}^k$. It is said to have no \textit{sources} if $v\Lambda^n\neq \emptyset$ for all $v\in \Lambda^0$ and $n\in \mathbb{N}^k$. If both the conditions are satisfied they we say that $\Lambda$ is \emph{row-finite without sources}. In this article we mainly focus on such $k$-graphs. There are other types of $k$-graphs defined in the literature by weakening the above conditions in various ways; for example \emph{locally convex} $k$-graphs \cite{Raeburn} and \emph{finitely aligned} $k$-graphs \cite{Yeend}.

Let $\Lambda$ be any $k$-graph and $\lambda\in \Lambda$. Suppose $m,n\in \mathbb{N}^k$ are such that $0\le m\le n\le d(\lambda)$. Note that $d(\lambda)=m+(n-m)+(d(\lambda)-n)$ and thus by applying the unique factorization, we have unique paths $\mu_1,\mu_2,\mu_3\in \Lambda$ such that 

\begin{center}
    $d(\mu_1)=m,d(\mu_2)=n-m, d(\mu_3)=d(\lambda)-n$ and $\lambda=\mu_1\mu_2\mu_3$.
\end{center}

Following the standard convention, we denote the paths $\mu_1,\mu_2$ and $\mu_3$ as $\lambda(0,m),\lambda(m,n)$ and $\lambda(n,d(\lambda))$ respectively. 

In some parts of the article we need some notions and terminologies which are often used for finitely aligned $k$-graphs. Let $\lambda,\mu\in \Lambda$. A path $\tau\in \Lambda$ is called a \emph{minimal common extension} of $\lambda$ and $\mu$ if 

\begin{center}
    $d(\tau)=d(\lambda)\vee d(\mu)$, $\tau(0,d(\lambda))=\lambda$ and $\tau(0,d(\mu))=\mu$.
\end{center}
Define 
\begin{center}
$\MCE(\lambda,\mu):=\{\tau~|~\tau$ is a minimal common extension of $\lambda,\mu\}$   
\end{center}
and 
\begin{center}
$\Lambda^{\min}(\lambda,\mu):=\{(\alpha,\beta)~|~\lambda\alpha=\mu\beta\in~ \MCE(\lambda,\mu)\}$.
\end{center}
Also for any subset $E\subseteq \Lambda$ and $\lambda\in \Lambda$, we denote
\begin{center}
Ext$(\lambda;E):=\displaystyle{\bigcup_{\mu\in E}} \{\alpha~|~(\alpha,\beta)\in \Lambda^{\min}(\lambda,\mu)$ for some $\beta\in \Lambda\}$.
\end{center}
Using the unique factorization one can see that the sets $\MCE(\lambda,\mu)$ and $\Lambda^{\min}(\lambda,\mu)$ are in bijection. The $k$-graph is called \emph{finitely aligned} if $\MCE(\lambda,\mu)$ is finite for all $\lambda,\mu\in \Lambda$. It is not hard to observe that any row-finite $k$-graph is finitely aligned. However the converse is not true; for example see \cite{Clark}. 

\subsubsection{Aperiodicity and cofinality} We now briefly recall the notion of infinite paths in a $k$-graph. For two $k$-graphs $(\Lambda,d_1)$ and $(\Omega,d_2)$, a $k$-\emph{graph morphism} $f:\Lambda\longrightarrow \Omega$ is a degree preserving functor (i.e., $d_2(f(\lambda))=d_1(\lambda)$ for all $\lambda\in \Lambda$). Any $k$-graph morphism $x:\Omega_k\longrightarrow \Lambda$ is called an \emph{infinite path} and the set 
\begin{center}
    $\Lambda^\infty:=\{x:\Omega_k\longrightarrow \Lambda~|~x$ is a $k$-graph morphism$\}$
\end{center}
is called the \emph{infinite path space} of $\Lambda$. For each $m\in \mathbb{N}^k$, we denote $x((m,m))$ simply as $x(m)$. Thus $x(m)$ is a vertex in $\Lambda$. We define the range of $x$ as $r(x):=x(0)$. For $v\in \Lambda^0$, $v\Lambda^\infty:=\{x\in \Lambda^\infty~|~x(0)=v\}$. 

For $p\in \mathbb{N}^k$ and $x\in \Lambda^\infty$, the $p$-shifted infinite path $\sigma^p(x)$ is defined as:
\begin{center}
    $\sigma^p(x)((m,n)):=x((m+p,n+p))$ for all $(m,n)\in \Omega_k$.
\end{center}

The notion of \textit{aperiodicity} of a $k$-graph was first introduced in \cite{Kumjian-Pask} as an analogue of Condition $(L)$ used in the study of Cuntz-Krieger algebras of directed graphs \cite{KPR}. Aperiodicity is the main hypothesis for the Cuntz-Krieger uniqueness theorems for higher-rank graph $C^*$-algebras and Kumjian-Pask algebras and consequently it is used to characterize the simplicity of these algebras.

\begin{dfn}\label{def KP convention}   
Let $\Lambda$ be a row-finite $k$-graph without sources. An infinite path $x\in \Lambda^\infty$ is said to be \textit{periodic} if there exists $0\neq p\in \mathbb{Z}^k$ such that $\sigma^n (x)=\sigma^{n+p} (x)$ for all $n\in \mathbb{N}^k$ with $n+p\in \mathbb{N}^k$. It is called \textit{aperiodic} if it has no nonzero period i.e., $\sigma^m (x)=\sigma^n (x)$ for $m,n\in \mathbb{N}^k$ implies $m=n$. The $k$-graph $\Lambda$ is said to satisfy the aperiodicity condition $(A)$ (or $\Lambda$ is \emph{aperiodic}) if for all $v\in \Lambda^0$ there is an aperiodic infinite path $x$ with $r(x)=x(0)=v$.
\end{dfn}

The concept of aperiodicity can also be defined in terms of finite paths. 
\begin{dfn}\label{def RS convention}
(\cite{Robertson}, \textbf{Robertson-Sims finite path formulation})

Let $\Lambda$ be a row-finite $k$-graph without sources. Then $\Lambda$ is called aperiodic if for each $v\in \Lambda^0$ and each pair of distinct elements $m,n\in \mathbb{N}^k$, there is a path $\lambda\in v\Lambda$ such that $d(\lambda)\ge m\vee n$ and
\[
\lambda(m,m+d(\lambda)-(m\vee n))\neq \lambda(n,n+d(\lambda)-(m\vee n)).
\]
\end{dfn}

In 2010, Lewin and Sims \cite{Lewin} gave another definition of aperiodicity using finite paths in the general setting of a finitely aligned $k$-graph. 

\begin{dfn}\label{def LS convention}
(\cite{Lewin}, \textbf{Lewin-Sims finite path formulation})

Let $\Lambda$ be a finitely aligned $k$-graph. Then $\Lambda$ is called aperiodic if for every pair of distinct paths $\alpha,\beta\in \Lambda$ with $s(\alpha)=s(\beta)$, there exists $\tau\in s(\alpha)\Lambda$ such that $\MCE(\alpha\tau,\beta\tau)=\emptyset$ (or equivalently, $\Lambda^{\min}(\alpha\tau,\beta\tau)=\emptyset$).
\end{dfn}

In case of row-finite $k$-graphs with no sources, all the above aperiodicity conditions are equivalent. We find Lewin-Sims criterion more easier to work with in our context. 

There is also a higher-rank generalization of the so called condition $(K)$ of directed graphs, known as \emph{strong aperiodicity} \cite{Kang-Pask}. For this we need the notion of hereditary and saturated subsets in a $k$-graph. A subset $H\subseteq \Lambda^0$ is called \emph{hereditary} if $r(\lambda)\in H$ for $\lambda\in \Lambda$ implies $s(\lambda)\in H$; on the other hand $H$ is called \emph{saturated} if for $v\in \Lambda^0$ and $n\in \mathbb{N}^k$, $s(v\Lambda^n)\subseteq H$ implies $v\in H$. The set of all hereditary saturated subsets of $\Lambda^0$ is denoted as $\mathcal{H}(\Lambda)$ and it forms a lattice under the partial order of inclusion.

For any $H\in \mathcal{H}(\Lambda)$, the \emph{quotient $k$-graph} of $\Lambda$ by $H$ (see \cite{Pino}) is defined as a $k$-graph with underlying set 
\begin{center}
    $\Lambda/H:=\{\lambda\in H~|~s(\lambda)\notin H\}$
\end{center}
and set of objects
\begin{center}
    $(\Lambda/H)^0:=\Lambda^0\setminus H$.
\end{center}
The range, source and the degree maps are the restrictions of $r,s$ and $d$ to $\Lambda/H$. It is easy to see that $\Lambda/H$ is also a row-finite $k$-graph without sources. The $k$-graph is called \emph{strongly aperiodic} (see \cite{Kang-Pask}) if $\Lambda/H$ is aperiodic for every proper hereditary saturated subset $H$ of $\Lambda^0$.

Another notion which plays significant role in characterizing simplicity of higher-rank graph $C^*$-algebras and Kumjian-Pask algebras is \emph{cofinality}. A $k$-graph $\Lambda$ is said to be \textit{cofinal} (see \cite[Definition 4.7]{Kumjian-Pask}) if for every $x\in \Lambda^\infty$ and every $v\in \Lambda^0$, there exists $n\in \mathbb{N}^k$ with $v\Lambda x(n)\neq \emptyset$.

\subsubsection{Kumjian-Pask algebras}
In 2013, Aranda Pino et al. \cite{Pino} defined Kumjian-Pask algebras as a higher-rank analogue of Leavitt path algebras. It was first defined for row-finite $k$-graphs without sources and later generalized for locally convex $k$-graphs by Clark, Flynn and an Huef \cite{Flynn} and for finitely aligned $k$-graphs by Clark and Pangalela \cite{Clark}. 

Suppose $\Lambda$ is a row-finite $k$-graph. Let $\Lambda^{\neq 0}:=\{\lambda\in \Lambda~|~d(\lambda)\neq 0\}=\Lambda\setminus \Lambda^0$ and set $G(\Lambda^{\neq 0}):=\{\lambda^*~|~\lambda\in \Lambda^{\neq 0}\}$. The symbols $\lambda^*$ are called \emph{ghost paths} and accordingly $G(\Lambda^{\neq 0})$ is the set of ghost paths. 
\begin{dfn}\label{def KP family}
(\cite[Definition 3.1]{Pino}, \textbf{Kumjian-Pask $\Lambda$-family}) 

Let $\Lambda$ be a row-finite $k$-graph without sources and $R$ a unital commutative ring. A \emph{Kumjian-Pask $\Lambda$-family} $(P,S)$ in an $R$-algebra $A$ consists of two maps $P:\Lambda^0\longrightarrow A$ and $S:\Lambda^{\neq 0}\cup G(\Lambda^{\neq 0})\longrightarrow A$ (where images are written as subscripts) such that the following hold:

$(KP1)$ $P_uP_v=\delta_{u,v}P_u$ for all $u,v\in \Lambda^0$;

$(KP2)$ $S_\lambda S_\mu=S_{\lambda\mu}$, $S_{\mu^*}S_{\lambda^*}=S_{(\lambda\mu)^*}$ for all $\lambda,\mu\in \Lambda^{\neq 0}$ with $s(\lambda)=r(\mu)$ and

$P_{r(\lambda)}S_\lambda=S_\lambda=S_\lambda P_{s(\lambda)}$, $P_{s(\lambda)}S_{\lambda^*}=S_{\lambda^*}=S_{\lambda^*}P_{r(\lambda)}$ for all $\lambda\in \Lambda^{\neq 0}$;

$(KP3)$ $S_{\lambda^*} S_\mu=\delta_{\lambda,\mu}P_{s(\lambda)}$ for all $\lambda,\mu\in \Lambda^{\neq 0}$ with $d(\lambda)=d(\mu)$;

$(KP4)$ $P_v=\displaystyle{\sum_{\lambda\in v\Lambda^n}} S_\lambda S_{\lambda^*}$ for all $v\in \Lambda^0$ and for all $n\in \mathbb{N}^k$.
\end{dfn}

Given a row-finite $k$-graph $\Lambda$ with no sources and a unital commutative ring $R$, the authors in \cite[Theorem 3.4]{Pino} constructed an $R$-algebra $\KP_R(\Lambda)$ generated by a Kumjian-Pask $\Lambda$-family $(p,s)$ which satisfies the following universal property: whenever $(Q,T)$ is any Kumjian-Pask $\Lambda$-family in any $R$-algebra $A$, then there is a unique $R$-algebra homomorphism $\pi_{Q,T}:$ $\KP_R(\Lambda)\longrightarrow A$ such that 
\begin{center}
$\pi_{Q,T}(p_v)=Q_v$, $\pi_{Q,T}(s_\lambda)=T_\lambda$ and $\pi_{Q,T}(s_{\mu^*})=T_{\mu^*}$
\end{center}
for all $v\in \Lambda^0$ and $\lambda,\mu\in \Lambda^{\neq 0}$. 

Moreover the algebra $\KP_R(\Lambda)$ is equipped with a natural $\mathbb{Z}^k$-grading where the $n^{\text{th}}$ homogeneous component is defined as
\begin{center}
    $\KP_R(\Lambda)_n=\lspan_R(\{s_\lambda s_{\mu^*}~|~\lambda,\mu\in \Lambda$ with $s(\lambda)=s(\mu)$ and $d(\lambda)-d(\mu)=n\})$
\end{center}
for all $n\in \mathbb{Z}^k$. The $\mathbb{Z}^k$-graded $R$-algebra $\KP_R(\Lambda)$ is known as the \emph{Kumjian-Pask algebra} of the $k$-graph $\Lambda$.

\subsection{Commutative monoids}
We recall some basic concepts related with commutative monoids. 

Let $(M,+)$ be a commutative monoid with identity $0$. One can define a natural pre-order $\le$ on $M$ as follows:
\begin{center}
    $x\le y\Longleftrightarrow y=x+z$ for some $z\in M$.
\end{center}
This relation is usually called the \emph{algebraic pre-order} on $M$. One can easily see that $\le$ is compatible with the addition of $M$ i.e., if $x\le y$ and $x'\le y'$ then $x+x'\le y+y'$. A submonoid $N$ of $M$ is called an \emph{order ideal} if it is a down set with respect to the algebraic pre-order i.e., if $x\le y\in N$ implies $x\in N$. The set of all order ideals of $M$ is denoted as $\mathcal{L}^o(M)$ and it forms a \emph{lattice} under set inclusion where the meet is the usual intersection and the join of two order ideals $N_1,N_2$ is defined as
\begin{center}
    $N_1\vee N_2:=\overline{N_1+N_2}=\{x\in M~|~x\le y_1+y_2$ for some $y_1\in N_1$ and $y_2\in N_2\}$.
\end{center}
In connection with the algebraic pre-order $\le$ on $M$, we have the following list of definitions which can be found in \cite{Bosa,APW,Vadim,Wehrung}.
\begin{dfns}\label{def concepts on monoids w.r.t the preorder}
$(i)$ An element $u\in M$ is called a \emph{unit} if $u\le 0$. 

$(ii)$ $M$ is called \emph{conical} if $0$ is the only unit in $M$ (equivalently, if $M^*:=M\setminus \{0\}$ is a semigroup). 

$(iii)$ $M$ is called \emph{antisymmetric} if the relation $\le$ is antisymmetric.

$(iv)$ A nonunit $a\in M$ is called an \emph{atom} if $a=b+c$ implies either $b$ is a unit or $c$ is a unit in $M$.

$(v)$ The monoid $M$ is called \emph{atomic} if each nonunit in $M$ can be written as a finite sum of atoms.

$(vi)$ A nonunit $x\in M$ is said to have a \emph{unique factorization into atoms} if $x$ can be written as a sum of atoms and whenever 
\begin{center}
    $x=a_1+a_2+...+a_m=a_1'+a_2'+...+a_n'$
\end{center}
where $a_1,a_2,...,a_m$, $a_1',a_2',...,a_n'$ are atoms then $n=m$ and there is a permutation $\pi\in S_n$ such that
\begin{center}
$a_i\le a_{\pi(i)}'$ and $a_{\pi(i)}'\le a_i$ for all $i=1,2,...,n$.     
\end{center}

$(vii)$ $M$ is called \emph{unique factorization monoid (UFM)} if $M$ is atomic and every nonunit in $M$ has a unique factorization into atoms.

$(viii)$ An element $a\in M$ is called \emph{minimal} if $b\le a$ for $b\in M$ implies either $b=0$ or $b=a$.
\end{dfns}

A commutative monoid $M$ is called a \emph{refinement monoid} if for all $a,b,c,d\in M$ with $a+b=c+d$ there exist $z_1,z_2,z_3,z_4\in M$ such that 
\begin{center}
    $a=z_1+z_2$, $b=z_3+z_4$ and $c=z_1+z_3$, $d=z_2+z_4$.
\end{center}
$M$ is called a \emph{cancellative monoid} if $x+y=x+z$ for $x,y,z\in M$ implies $y=z$. 
\begin{rmk}\label{rem consequences in conical cancellative monoid}
In a conical monoid $M$, $0$ is a minimal element and $\{0\}$ is an order ideal. Moreover if $M$ is cancellative then it is antisymmetric and in this case atoms and nonzero minimal elements do coincide.    
\end{rmk}
Given a commutative monoid $M$ and a submonoid $N$, one can define a relation $\equiv_N$ on $M$ by 
\begin{center}
    $x\equiv_N y \Longleftrightarrow x+a=y+b$ for some $a,b\in N$
\end{center}
It is straightforward to check that $\equiv_N$ is a monoid congruence on $M$. We sometimes write the quotient monoid $M/\equiv_N$ as $M/N$. 

\subsubsection{$\Gamma$-monoid} 
Given a group $\Gamma$ and a commutative monoid $M$, we say that $M$ is a \emph{$\Gamma$-monoid} if $\Gamma \curvearrowright M$ ($\Gamma$ acts on $M$) by monoid automorphisms. Following the standard notation used in \cite{Cordeiro, Hazrat}, we denote the action of $\alpha\in \Gamma$ on $x\in M$ by $^{\alpha}x$. Note that the $\Gamma$-action is compatible with the algebraic pre-ordering in the sense that $x\le y$ implies $^{\alpha}x~\le ~^{\alpha}y$ for all $\alpha\in \Gamma$. 

If $M_1$ and $M_2$ are $\Gamma$-monoids then a \emph{$\Gamma$-monoid homomorphism} $f:M_1\longrightarrow M_2$ is a monoid homomorphism which respects the $\Gamma$-action i.e., $f(^{\alpha}x)=~^{\alpha}f(x)$ for all $\alpha\in \Gamma$ and $x\in M$. 

An order ideal $J$ of a $\Gamma$-monoid $M$ is said to be a \emph{$\Gamma$-order ideal} if $J$ is closed under the $\Gamma$-action. Equivalently a $\Gamma$-order ideal $J$ is a subset of $M$ such that for any $\alpha,\beta\in \Gamma$ and $x,y\in M$, $^{\alpha}x+~^{\beta}y\in J$ if and only if $x,y\in J$. We denote the lattice of all $\Gamma$-order ideals of $M$ by $\mathcal{L}(M)$. For a subset $A$ of $M$, $\langle A\rangle$ denotes the $\Gamma$-order ideal generated by $A$. A nonzero conical $\Gamma$-monoid $M$ is called \emph{simple} if $\{0\}$ and $M$ are the only $\Gamma$-order ideals of $M$. A nonzero element $a\in M$ is called \textit{periodic} if there exists $\alpha\neq 0 \in \Gamma$ such that $^{\alpha}a=a$, otherwise $a$ is called \emph{aperiodic}. We say that $\Gamma$ acts \emph{freely} on $M$ if there is no periodic element in $M$. 

A $\Gamma$-order ideal $P$ of $M$ is called \emph{prime} if for any two $\Gamma$-order ideals $J_1,J_2$ of $M$, $J_1\cap J_2\subseteq P$ implies either $J_1\subseteq P$ or $J_2\subseteq P$. 

A $\Gamma$-monoid $M$ is called \emph{finitely generated as a $\Gamma$-monoid} if there is a finite subset $A\subseteq M$ such that for any $x\in M$ there exist $\alpha_1,\alpha_2,...,\alpha_n\in \Gamma$ and $a_1,a_2,...,a_n\in A$ (not necessarily distinct) such that $x=\displaystyle{\sum_{i=1}^{n}} ~^{\alpha_i}a_i$. Clearly in this case $M=\langle A\rangle$.

\subsubsection{Talented monoid of a directed graph}

As we seek to cast the analogue of talented monoid for higher-rank graphs, we include its definition in the setting of directed graphs.

\begin{dfn}\label{def talented monoid}
(\cite{Hazrat,Cordeiro}) Let $E$ be any row-finite directed graph. The \emph{talented monoid} of $E$ is denoted as $T_E$ and is defined to be the commutative monoid generated by $\{v(i)~|~v\in E^0,i\in \mathbb{Z}\}$ subject to the relations:

\begin{center}
$v(i)=\displaystyle{\sum_{e\in s^{-1}(v)}} r(e)(i+1)$
\end{center}
for all vertices $v$ which are not sink and for all $i\in \mathbb{Z}$.
\end{dfn}
The monoid $T_E$ is a $\mathbb{Z}$-monoid with respect to the $\mathbb{Z}$-action defined on generators as: 
\begin{center}
    $^n v(i):=v(i+n)$.
\end{center}

This monoid can be seen as a graded version of the graph monoid $M_E$. To witness the beautiful relations of the talented monoid with the geometry of the underlying directed graph as well as with the algebraic properties of the associated Leavitt path algebra, the readers may go through \cite{Hazrat, Cordeiro}.

\subsection{Groupoid and Type semigroup}
If $\Lambda$ is a $k$-graph then using the degree functor $d$ together with the fact that $\mathbb{N}^k$ is a conical monoid, one can verify that the only arrows of $\Lambda$ which are invertible are the units i.e., no $\lambda\in \Lambda\setminus \Lambda^0$ is invertible. In the world of small categories, \emph{groupoids} sit completely at the other corner. A \emph{groupoid} is a small category consisting of isomorphisms only. We assume that the readers are familiar with the basic notions of topological groupoids.

\subsubsection{The path groupoid of a $k$-graph}\label{ss infinite path groupoid}
Let $\Lambda$ be a row-finite $k$-graph with no sources. In \cite{Kumjian-Pask}, the authors defined a groupoid $\mathcal{G}_\Lambda$ to obtain a groupoid model for the higher-rank graph $C^*$-algebra $C^*(\Lambda)$. We briefly describe the construction.

Define 
\begin{center}
    $\mathcal{G}_\Lambda:=\{(x,n-m,y)\in \Lambda^\infty \times \mathbb{Z}^k \times \Lambda^\infty~|~n,m\in \mathbb{N}^k$ and $\sigma^n (x)=\sigma^m(y)\}$.
\end{center}
The unit space $\mathcal{G}_\Lambda^{(0)}:=\Lambda^\infty$; the range and source maps are defined as $r((x,n,y)):=x$ and $s((x,n,y)):=y$. Two arrows $(x,n,y)$ and $(z,m,w)$ are composable if $y=z$ and in that case
\begin{center}
    $(x,n,y)(y,m,w):=(x,n+m,w)$.
\end{center}
The inverse of an arrow $(x,n,y)$ is defined as $(x,n,y)^{-1}:=(y,-n,x)$. Then $\mathcal{G}_\Lambda$ is a groupoid which is known as the \emph{(infinite) path groupoid} of $\Lambda$. 

For $\lambda,\mu\in \Lambda$ with $s(\lambda)=s(\mu)$ define
\begin{center}
    $Z(\lambda,\mu):=\{(\lambda z, d(\lambda)-d(\mu), \mu z)~|~z\in s(\lambda)\Lambda^\infty\}$.
\end{center}
Then $\mathcal{B}:=\{Z(\lambda,\mu)~|~ \lambda,\mu\in \Lambda, s(\lambda)=s(\mu)\}$ is a basis for a locally compact, Hausdorff topology on $\mathcal{G}_\Lambda$. Moreover each $Z(\lambda,\mu)$ is a compact open bisection which makes $\mathcal{G}_\Lambda$ an ample groupoid. If we identify $x\in \Lambda^\infty$ with the arrow $(x,0,x)$ then the sets $Z(\lambda):=Z(\lambda,\lambda)$ form a basis of compact open sets for the subspace topology on $\mathcal{G}_\Lambda^{(0)}=\Lambda^\infty$. 

If $\Lambda=E^*$ (the path category of a row-finite directed graph $E$ with no sink), then $\mathcal{G}_{E^*}\cong \mathcal{G}_E$, the boundary path groupoid of $E$. Both the groupoids $\mathcal{G}_E$ and $\mathcal{G}_\Lambda$ are prototypes of the so-called \emph{Deaconu-Renault groupoids} \cite{Deaconu}.

\subsubsection{Type semigroup}
Let $S$ be an inverse semigroup with $0$. Two elements $x,y\in S$ are said to be \emph{orthogonal} if $x^*y=xy^*=0$. If $x,y$ are orthogonal then we write $x\perp y$. The semigroup $S$ is called a \emph{Boolean inverse semigroup} if the semilattice of idempotents $E(S)$ is a Boolean ring (equivalently a \emph{generalized Boolean algebra}) such that any two orthogonal elements $x,y$ have a \emph{join} which is denoted by $x\oplus y$. 
\begin{dfn}\label{def type semigroup}
(\cite{Bosa,Cordeiro,Wehrung}) Let $S$ be a Boolean inverse semigroup. For each $x\in E(S)$ we set a symbol $\typ(x)$. The \emph{type semigroup} or \emph{type monoid} of $S$ is denoted as $\Typ(S)$ and defined as the commutative monoid freely generated by the symbols $\typ(x)$, where $x\in E(S)$ subject to the following relations:

$(i)$ $\typ(0)=0$;

$(ii)$ $\typ(x)=\typ(y)$ whenever there exists $s\in S$ such that $x=s^*s$ and $y=ss^*$;

$(iii)$ $\typ(x\oplus y)=\typ(x)+\typ(y)$ whenever $x,y\in E(S)$ are such that $x\perp y$ and $x\oplus y\in E(S)$.
\end{dfn}

An interesting fact about type monoid is that it is always a conical refinement monoid (see \cite[Corollary 4.1.4]{Wehrung}). Let $\mathcal{G}$ be a (not necessarily Hausdorff) ample groupoid. We denote $\mathcal{G}^a$ to be the collection of all compact open bisections in $\mathcal{G}$. Then $\mathcal{G}^a$ is an inverse semigroup with $0$ (the empty bisection) where the operations are defined as:
\begin{center}
    $AB:=\{\alpha\beta~|~(\alpha,\beta)\in \mathcal{G}^{(2)}\cap (A\times B)\}$

    $A^{-1}:=\{\alpha^{-1}~|~\alpha\in A\}$
\end{center}
for all $A,B\in \mathcal{G}^a$.

The natural partial order on this inverse semigroup becomes the usual set inclusion. The semilattice of idempotents $E(\mathcal{G}^a)$ coincides with the set of all compact open subsets of the unit space $\mathcal{G}^{(0)}$ and it is a Boolean ring where sum and product are respectively the symmetric difference and the intersection. Moreover if $A,B\in \mathcal{G}^a$ are such that $A\perp B$ (that means $A^{-1}B=AB^{-1}=\emptyset$) then $A\cap B=\emptyset$ and in this case $A\oplus B:=A \sqcup B$ is the join of $A,B$. This shows that $\mathcal{G}^a$ is a Boolean inverse semigroup. One of the important examples of type semigroups is the \emph{type semigroup of $\mathcal{G}$} which is defined as $\Typ(\mathcal{G}):=\Typ(\mathcal{G}^a)$.

Type semigroup of an ample groupoid is defined in some other ways also, for example see \cite[Definition 5.3]{Bonicke} and \cite[Definition 5.4]{Rainone}. However all these definitions are equivalent. We briefly recall B\"{o}nicke and Li's definition from \cite{Bonicke} for its use in this article. 

Suppose $\mathcal{G}$ is an ample groupoid. Define a relation $\sim$ on 
\begin{center}
    $\left\{\displaystyle{\bigcup_{i=1}^{n}}~ (A_i\times \{i\})~|~n\in \mathbb{N},A_i\in E(\mathcal{G}^a)\right\}$
\end{center}
as follows: for two such elements $A=\displaystyle{\bigcup_{i=1}^{n}}~ (A_i\times \{i\})$ and $B=\displaystyle{\bigcup_{j=1}^{m}}~ (B_j\times \{j\})$, $A\sim B$ if there exist $l\in \mathbb{N}$, $W_1,W_2,...,W_l\in \mathcal{G}^a$ and $n_1,n_2,...,n_l,m_1,m_2,...,m_l\in \mathbb{N}$ such that 
\begin{center}
   $A=\displaystyle{\bigsqcup_{k=1}^{l}} ~s(W_k)\times \{n_k\}$ and $B=\displaystyle{\bigsqcup_{k=1}^{l}} ~r(W_k)\times \{m_k\}$. 
\end{center}
It is an equivalence relation as shown in \cite[Lemma 5.2]{Bonicke}. The \emph{type semigroup} of $\mathcal{G}$ is defined as the set 
\begin{center}
    $S(\mathcal{G},\mathcal{G}^a):=\left\{\displaystyle{\bigcup_{i=1}^{n}}~ (A_i\times \{i\})~|~n\in \mathbb{N},A_i\in E(\mathcal{G}^a)\right\}/\sim$,
\end{center}
where the operation is defined as 
\begin{center}
$\left[\displaystyle{\bigcup_{i=1}^{n}}~ (A_i\times \{i\})\right]+\left[\displaystyle{\bigcup_{j=1}^{m}}~ (B_j\times \{j\})\right]:=\left[\displaystyle{\bigcup_{i=1}^{n}}~ (A_i\times \{i\}) \cup \displaystyle{\bigcup_{j=1}^{m}}~ (B_j\times \{n+j\})\right]$.
\end{center}
In \cite[Proposition 7.3]{Bosa}, it was shown that $\Typ(\mathcal{G})\cong S(\mathcal{G},\mathcal{G}^a)$.

\subsubsection{Skew-product groupoid and graded type semigroup}
Let $\Gamma$ be a discrete commutative group and $\mathcal{G}$ ample Hausdorff groupoid which is $\Gamma$-graded by a cocycle $c:\mathcal{G}\longrightarrow \Gamma$. The \emph{skew-product} of $\mathcal{G}$ by $\Gamma$ (see \cite[Definition 1.6]{Renault}) is denoted as $\mathcal{G}\times_c \Gamma$ and it is a groupoid with set of arrows $\mathcal{G}\times \Gamma$ and unit space $\mathcal{G}^{(0)}\times \Gamma$. For $(x,\alpha)\in \mathcal{G}\times_c \Gamma$ define $r((x,\alpha)):=(r(x),\alpha)$ and $s((x,\alpha)):=(s(x),\alpha c(x))$. Two arrows $(x,\alpha),(y,\beta)$ are composable if $(x,y)\in \mathcal{G}^{(2)}$ and $\beta=\alpha c(x)$; in this case $(x,\alpha)(y,\beta):=(xy,\alpha)$. The inverse is defined as $(x,\alpha)^{-1}:=(x^{-1},\alpha c(x))$. 

Note that $\mathcal{G}\times_c \Gamma$ is also an ample Hausdorff groupoid with respect to the product topology. The semilattice of idempotent $E((\mathcal{G}\times_c \Gamma)^a)$ is the set of all compact open subsets of $\mathcal{G}^{(0)}\times \Gamma$. Since $\mathcal{G}^{(0)}\times \Gamma$ is Hausdorff, one can write each such compact open subset as a disjoint union of sets of the form $U\times \{\alpha\}$ where $U$ is a compact open subset of $\mathcal{G}^{(0)}$ and $\alpha\in \Gamma$.
\begin{dfn}\label{def graded type semigroup}
(\cite[Definition 3.3]{Cordeiro}) Let $\Gamma$ be a discrete commutative group and $\mathcal{G}$ a $\Gamma$-graded ample Hausdorff groupoid. The \emph{graded type semigroup} of $\mathcal{G}$ is defined as
\begin{center}
    $\Typ^{\gr}(\mathcal{G}):=\Typ(\mathcal{G}\times_c \Gamma)=\Typ((\mathcal{G}\times_c \Gamma)^a)$.
\end{center}
\end{dfn}
The graded type semigroup $\Typ^{\gr}(\mathcal{G})$ is canonically a $\Gamma$-monoid where the $\Gamma$-action is defined on the generators by
\begin{center}
    $^{\gamma}\typ(U\times \{\alpha\}):=\typ(U\times \{\gamma\alpha\})$
\end{center}
and extended linearly. In this paper we denote the graded type semigroup of $\mathcal{G}$ by $T_{\mathcal{G}}$ and call it the \emph{talented monoid} of $\mathcal{G}$. It is established in \cite{Cordeiro} that the talented monoid of the graph groupoid $\mathcal{G}_E$ of a directed graph $E$ is same as the talented monoid of the graph $E$.

\section{The talented monoid of a higher-rank graph}\label{sec the talented monoid}
In this section we first seek a higher-rank graph analogue of the refinement monoid associated to a directed graph in the context of Leavitt path algebras~\cite{Abrams-Monograph, Ara}. In the context of directed graphs, this monoid is defined as the quotient of the free commutative monoid on the vertices of the graph  by relations that each vertex $v$ is equal to the the sum of vertices $v$ connects to (equation (M) in \cite{Ara})
which resembles the Cuntz-Krieger relation. To adapt this definition to the context of higher-rank graphs we must first alter the sum to be over  those vertices connecting \emph{to} $v$, second we must create a sum for those vertices which connect to $v$ \emph{by paths of each dimension}, in analogy with (KP4) of Definition \ref{def KP family}. 

Throughout this section we fix $\Lambda$ to be a row-finite $k$-graph.

\subsection{The $k$-graph monoid}
Let $\mathbb{F}_{_\Lambda}:=\mathbb{N}\Lambda^0$, the free commutative monoid generated by $\Lambda^0$. For each $i=1,2,..., k$ we define a binary relation $\longrightarrow_i$ on $\mathbb{F}_{_\Lambda}\setminus \{0\}$ as follows:

For $\displaystyle{\sum_{j=1}^{m}} x_j\in \mathbb{F}_{_\Lambda}$ and for some $t\in \{1,2,...,m\}$,

\begin{center}
    $\displaystyle{\sum_{j=1}^{m}} x_j\longrightarrow_i \displaystyle{\sum_{\substack{1\le j \le m\\ j\neq t}}} x_j +\mathfrak{s}^i(x_t)$ if $x_t\Lambda^{e_i}\neq \emptyset$
\end{center}
\begin{center}
$\mathfrak{s}^i(x_t):=\displaystyle{\sum_{\lambda\in x_t \Lambda^{e_i}}} s(\lambda)$. 
\end{center}
Now we define a binary relation $\longrightarrow$ on $\mathbb{F}_{_\Lambda}\setminus \{0\}$ by $\alpha\longrightarrow \beta$ if either $\alpha=\beta$ or there exists a finite string 
\begin{center}
    $\alpha=\alpha_0\longrightarrow_{i_1}\alpha_1\longrightarrow_{i_2}\alpha_2\longrightarrow_{i_3}\cdot\cdot\cdot\longrightarrow_{i_N}\alpha_N=\beta$
\end{center}
where $\alpha_j\in \mathbb{F}_{_\Lambda}\setminus \{0\}$ and $i_j\in \{1,2,...,k\}$ for all $j=1,2,...,N$.

Let $\sim$ be the $0$-restricted congruence on $\mathbb{F}_{_\Lambda}$ generated by the $k$-relations $\longrightarrow_i$ or equivalently by the single relation $\longrightarrow$. We define the \emph{$k$-graph monoid} $M_\Lambda:=\mathbb{F}_{_\Lambda}/\sim$. Since $\sim$ is $0$-restricted, it follows that $M_\Lambda$ is a conical monoid.

The following lemma is a $k$-graph analogue of \cite[Lemma 3.6.6]{Abrams-Monograph} and \cite[Lemma 4.2]{Ara}.
\begin{lem}\label{lem the splitting lemma}
Let $\alpha,\beta,\alpha_1,\alpha_2\in \mathbb{F}_{_\Lambda}\setminus \{0\}$ be such that $\alpha=\alpha_1+\alpha_2$ and $\alpha\longrightarrow \beta$. Then there exist $\beta_1,\beta_2\in \mathbb{F}_{_\Lambda}\setminus \{0\}$ such that $\beta=\beta_1+\beta_2$ and $\alpha_i\longrightarrow \beta_i$ for $i=1,2$.    
\end{lem}
\begin{proof}
Since $\longrightarrow$ is the reflexive and transitive closure of the relations $\longrightarrow_i$, $i=1,2,...,k$; it suffices to prove the result where $\alpha\longrightarrow_i \beta$ for some $i\in \{1,2,...,k\}$. Since $\alpha \longrightarrow_i \beta$ then there is an element $v \in \Lambda^0$ in the support of $\alpha$ such that $\beta = ( \alpha - v ) + \mathfrak{s}^i(v)$. The element $v$ either belongs to the support of $\alpha_1$ or $\alpha_2$. Assume, for instance that the element $v$ belongs to the support of $\alpha_1$. We set $\beta_1 = ( \alpha_1 - v ) + \mathfrak{s}^i(v)$ and $\beta_2 = \alpha_2$. Then $\alpha_1\longrightarrow_i \beta_1$ and $\alpha_2\longrightarrow_i \beta_2$.
\end{proof}
It is interesting to know whether there is a $k$-graph $\Lambda$ whose monoid $M_\Lambda$ can not be realized as the graph monoid of any directed graph $E$. We provide an example below.
\begin{example}\label{ex k-graph monoid which is not a graph monoid}
Consider $\Lambda$ to be the $3$-graph with the following $1$-skeleton:

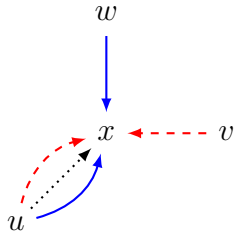
\[
\begin{tikzpicture}[scale=0.8]
\node[inner sep=2.5pt, circle] (27) at (0,0) {$x$};
\node[inner sep=2.5pt, circle] (28) at (0,2) {$w$};	
\node[inner sep=2.5pt, circle] (29) at (2,0) {$v$};
\node[inner sep=2.5pt, circle] (30) at (-1.5,-1.5) {$u$};

\path[->, red, dashed, >=latex,thick] (29) edge [left] node {} (27);
\path[->, blue, >=latex, thick] (28) edge [below] node {} (27);
\path[->, red, dashed, >=latex,thick] (30) edge [bend left] node {} (27);
\path[->, blue, >=latex,thick] (30) edge [bend right] node {} (27);
\path[->,dotted, >=latex,thick] (30) edge [] node {} (27);
\end{tikzpicture}
\]
\noindent 
where solid blue edges are of degree $(1,0,0)$, dashed red edges are of degree $(0,1,0)$,  and the dotted black edge has degree $(0,0,1)$. We have the following relations in $M_\Lambda$.
\begin{align*}
    x= & ~u ~(\text{since}~x\longrightarrow_3 u);\\
    x= & ~u+v ~(\text{since}~x\longrightarrow_2 u+v);\\
    x= & ~u+w ~(\text{since}~x\longrightarrow_1 u+w).
\end{align*}
\noindent
Effectively we have three generators and  $M_\Lambda=\langle u,v,w~|~u=u+v=u+w\rangle$. This is the monoid $M_0$ constructed by Ara et al. in \cite{APW}. It follows that $M_\Lambda$ is a finitely generated, commutative refinement monoid but it is not the graph monoid of any row-finite directed graph; for details see \cite[\S 3, \S 4 and Lemma 4.1]{APW}. 
\end{example}
\begin{rmk}\label{rem confluence may not follow}
Note that in the $3$-graph of above example, the vertices $u,v,w$ are sources. So if there is some $z\in \mathbb{F}_{_\Lambda}\setminus \{0\}$ such that $u+v\longrightarrow z$ and $u+w\longrightarrow z$, then we must have $z=u+v=u+w$ in $\mathbb{F}_{_\Lambda}\setminus \{0\}$ which is not possible in any free commutative monoid. Hence there is no such $z$, although we have $u+v\sim u+w$ in $\mathbb{F}_{_\Lambda}\setminus \{0\}$.     
\end{rmk}

The above remark says that unlike graph monoids of row-finite directed graphs, we may not have confluence property in the $k$-graph monoid for a row-finite $k$-graph which has sources. Therefore, for the remaining part of this section, we consider row-finite $k$-graphs without sources. In order to prove that we have certain confluence property in the monoid $M_\Lambda$, we need the following fact about the binary operation $\longrightarrow$ defined above. 

\begin{lem}\label{lem mini confluence}
Let $\Lambda$ be a row-finite $k$-graph with no sources. For any $v\in \Lambda^0$ and $i,j\in \{1,2,...,k\}$,  we have the following confluence 
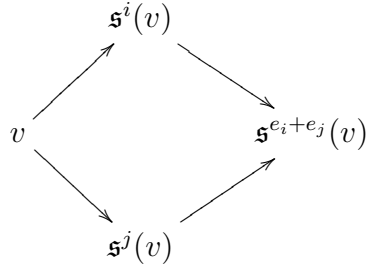
\begin{equation*}
\xymatrix{
&   \mathfrak{s}^{i}(v) \ar[dr] & \\
 v \ar[ur] \ar[dr] & & \mathfrak{s}^{e_i+e_j}(v)\\
& \mathfrak{s}^{j}(v) \ar[ur] & 
}
\end{equation*}
where $\mathfrak{s}^{e_i+e_j}(v):=\displaystyle{\sum_{\lambda\in v\Lambda^{e_i+e_j}}} s(\lambda)$.
\end{lem}
\begin{proof}
Clearly $v\longrightarrow_i \mathfrak{s}^i (v)$ and $v\longrightarrow_j \mathfrak{s}^j (v)$. We just need to show that $\mathfrak{s}^i (v), \mathfrak{s}^j(v)\longrightarrow \displaystyle{\sum_{\lambda\in v\Lambda^{e_i+e_j}}} s(\lambda)$. We first show that in the free commutative monoid $\mathbb{F}_{_\Lambda}$,

\begin{equation}\label{timefleet}
\sum_{\alpha \in v \Lambda^{m}}  \sum_{\beta \in s(\alpha) \Lambda^{n}} s(\beta)  = \sum_{\beta \in v \Lambda^{n}}  \sum_{\alpha \in s(\beta) \Lambda^{m}} s(\alpha) = \sum_{\alpha \in v \Lambda^{m+n}} s(\alpha).   
\end{equation}

We have 
\begin{equation*}
\sum_{\alpha \in v \Lambda^{m}}  \sum_{\beta \in s(\alpha) \Lambda^{n}} s(\beta) = 
\sum_{\alpha \in v \Lambda^{m}}  \sum_{\beta \in s(\alpha) \Lambda^{n}} s(\alpha \beta) =
\sum_{\alpha \beta\in A} s(\alpha \beta), 
\end{equation*}
where $A=\{ \alpha \beta \mid d(\alpha)=m, d(\beta)=n, r(\alpha)=v, s(\alpha)=r(\beta)  \}$. The latter equality follows from the fact that $d$ has the unique factorisation property and thus $\alpha \beta \not = \alpha' \beta'$ when $\alpha \not = \alpha'$. Then for $\alpha \beta \in A$, writing $d(\alpha \beta)=m+n=n+m$, the factorisation gives that $\alpha \beta = \beta' \alpha'$, where $\beta' \alpha' \in B$, where $B=\{ \beta \alpha \mid d(\beta)=n, d(\alpha)=m, r(\beta)=v, s(\beta)=r(\alpha) \}$. It follows that $A=B$. Therefore 
\[\sum_{\alpha \beta\in A} s(\alpha \beta) = \sum_{\beta \alpha\in B} s(\beta \alpha) = \sum_{\beta \in v \Lambda^{n}}  \sum_{\alpha \in s(\beta) \Lambda^{m}} s(\beta \alpha) =  \sum_{\beta \in v \Lambda^{n}}  \sum_{\alpha \in s(\beta) \Lambda^{m}} s(\alpha).\]

Furthermore setting  $C=\{ \alpha \mid d(\alpha)=m+n, r(\alpha)=v  \}$, factorisation guarantees that $A=B=C$ and the last equality of (\ref{timefleet}) follows. 

Now using (\ref{timefleet}) we have,
\begin{center}
    $\mathfrak{s}^i(v)=\displaystyle{\sum_{\alpha\in v\Lambda^{e_i}}} s(\alpha)\longrightarrow \sum_{\alpha \in v \Lambda^{e_i}}  \sum_{\beta \in s(\alpha)\Lambda^{e_j}} s(\beta)=  \sum_{\lambda \in v \Lambda^{e_i+e_j}} s(\lambda)$.
\end{center}

Similarly one can show that $\mathfrak{s}^j(v)\longrightarrow \displaystyle{\sum_{\lambda\in v\Lambda^{e_i+e_j}}} s(\lambda)$ and the lemma follows.
\end{proof}

Now we are ready to prove our desired Confluence lemma for the $k$-graph monoids of row-finite $k$-graphs without sources.
\begin{lem}\label{lem the confluence lemma}
(\textbf{Confluence lemma}) Let $\Lambda$ be a row-finite $k$-graph without sources. Let $\alpha,\beta\in \mathbb{F}_{_\Lambda}\setminus \{0\}$. Then $\alpha\sim \beta$ if and only if there exists $\gamma\in \mathbb{F}_{_\Lambda}\setminus \{0\}$ such that $\alpha\longrightarrow \gamma$ and $\beta \longrightarrow \gamma$.  
\end{lem}
\begin{proof}
The sufficiency is clear from the definition of the relation $\sim$. So assume that $\alpha\sim \beta$. Then there exists a finite sequence $\alpha=\alpha_0,\alpha_1,...,\alpha_n=\beta$ in $\mathbb{F}_{_\Lambda}\setminus \{0\}$ such that for each $t=0,1,...,n-1$ either $\alpha_t\longrightarrow_i \alpha_{t+1}$ or $\alpha_{t+1}\longrightarrow_j \alpha_t$ for some $i,j\in \{1,2,...,k\}$. We use induction on the length of the sequence. If $n=0$ then $\alpha=\beta$ then taking $\gamma=\alpha$ we are done. Assume that the result is true for any sequence of length $n-1$. Then applying the induction hypothesis on the sequence $\alpha=\alpha_0,\alpha_1,...\alpha_{n-1}$, we can argue that there exists $\lambda\in \mathbb{F}_{_\Lambda}\setminus \{0\}$ such that $\alpha\longrightarrow \lambda$ and $\alpha_{n-1}\longrightarrow \lambda$. The following cases may appear. 

\textit{Case-1:} If $\beta=\alpha_n\longrightarrow_j\alpha_{n-1}$ for some $j\in \{1,2,...,k\}$ then $\beta\longrightarrow \lambda$ and taking $\gamma=\lambda$ we are done.

\textit{Case-2:} Suppose $\alpha_{n-1}\longrightarrow_i \alpha_n=\beta$ for some $i\in \{1,2,...,k\}$. Since $\alpha_{n-1}\neq 0$, there exists $x\in \Lambda^0$ in the basic representation of $\alpha_{n-1}$ such that $\alpha_{n-1}=x+\alpha_{n-1}'$ and $\beta=\alpha_{n-1}'+\mathfrak{s}^i(x)$. Since $\alpha_{n-1}\longrightarrow \lambda$ so by Lemma \ref{lem the splitting lemma}, we can write $\lambda=\lambda(x)+\lambda'$ where $x\longrightarrow \lambda(x)$ and $\alpha_{n-1}'\longrightarrow \lambda'$.

\textit{Subcase-1:} $x=\lambda(x)$. Define $\gamma:=\mathfrak{s}^i(x)+\lambda'$. Then \[\alpha\longrightarrow \lambda=\lambda(x)+\lambda'\longrightarrow_i \mathfrak{s}^i(x)+\lambda'=\gamma\] and \[\beta=\alpha_{n-1}'+\mathfrak{s}^i(x)\longrightarrow \lambda'+\mathfrak{s}^i(x)=\gamma.\]

\textit{Subcase-2:} $x\neq \lambda(x)$. Suppose we have a sequence of length $m\ge 1$
\begin{center}
    $x\longrightarrow_{j_1}\eta_1\longrightarrow_{j_2}\eta_2\longrightarrow_{j_3}\cdot\cdot\cdot\longrightarrow_{j_m}\eta_m=\lambda(x)$.
\end{center}
If $j_1=i$ then clearly $\eta_1=\mathfrak{s}^i(x)$ and so $\mathfrak{s}^i(x)\longrightarrow \lambda(x)$. Therefore $\beta\longrightarrow \lambda(x)+\lambda'=\lambda$ and we are done. So assume that $j_1\neq i$. Using Lemma \ref{lem mini confluence} repeatedly and tactfully, we can obtain elements $\xi_1,\xi_2,...,\xi_m\in \mathbb{F}_{_\Lambda}\setminus \{0\}$ such that $\mathfrak{s}^i(x),\eta_1\longrightarrow \xi_1$ and $\xi_i,\eta_{i+1}\longrightarrow \xi_{i+1}$ for all $i=1,2,...,m-1$. Finally let $\gamma=\xi_m+\lambda'$. Then \[\alpha\longrightarrow \lambda=\lambda(x)+\lambda'\longrightarrow \xi_m+\lambda'=\gamma\] and \[\beta=\alpha_{n-1}'+\mathfrak{s}^i(x)\longrightarrow \lambda'+\xi_m=\gamma.\] 

\end{proof}
Before we give another realization of the $k$-graph monoid $M_\Lambda$, we need one more lemma. 
\begin{lem}\label{lem two sum equivalent}
Let $\Lambda$ be a row-finite $k$-graph without sources. Let $u\in \Lambda^0$ and $n\in \mathbb{N}^k$. Then in $\mathbb{F}_{_\Lambda}\setminus \{0\}$, $u\longrightarrow \displaystyle{\sum_{\lambda\in u\Lambda^n}} s(\lambda)$ and so $\displaystyle{\sum_{\lambda\in u\Lambda^n}} s(\lambda)~\sim \displaystyle{\sum_{\mu\in u\Lambda^m}} s(\mu)$ for all $n,m\in \mathbb{N}^k$.
\end{lem}
\begin{proof}
Let $n=(n_1,\dots,n_k)$. We argue by induction on $l(n):= n_1+\dots+n_k$.  If $l(n)=1$, then $n=e_i$ for some $1\leq i \leq k$. Then $v\longrightarrow_i \mathfrak{s}^i(u)=\displaystyle{\sum_{\alpha\in u\Lambda^{e_i}}} s(\alpha)$ follows from the definition of $\longrightarrow_i$. Suppose $l(n)=l>1$.  Then $n = n' + e_j$ for some $n'\in \mathbb N^k$ with $l(n')=l-1$ and some $1\leq j \leq k$. By induction 
\[u\longrightarrow\sum_{\alpha \in u \Lambda^{n'}} s(\alpha).\] 
Since $\Lambda$ is row-finite, there is a finite number of $s(\alpha)$ appearing in the above sum. Applying the binary operation $\longrightarrow_j$ on each of the these vertices, we have  $s(\alpha)\longrightarrow_j \displaystyle{\sum_{\beta \in s(\alpha)\Lambda^{e_j}} }s(\beta)$. Finally using Equation (\ref{timefleet}), we have 
\[u\longrightarrow\sum_{\alpha \in u\Lambda^{n'}} s(\alpha) \longrightarrow  \sum_{\alpha \in u \Lambda^{n'}}  \sum_{\beta \in s(\alpha)\Lambda^{e_j}} s(\beta)=  \sum_{\lambda \in u \Lambda^{n}} s(\lambda).\]
\end{proof}
Applying the above lemma, it is now easy to prove the following:

\begin{prop}\label{pro another realization of k-graph monoid}
Let $\Lambda$ be a row-finite $k$-graph without sources. Then $k$-graph monoid $M_\Lambda$ is the quotient of the free commutative monoid $\mathbb{N}\Lambda^0$ modulo the congruence generated by the set of relations
$$\Big \{(u,\displaystyle{\sum_{\lambda\in u\Lambda^n}} s(\lambda))~|~u\in \Lambda^0,n\in \mathbb{N}^k\Big\}.$$ 
\end{prop}

\begin{example}\label{ghghgd}
(\cite[Example 5.7]{Kang-Pask})
Suppose that $\Lambda$ is the following 2-graph, with single vertex $v$ and the $1$-skeleton shown below.
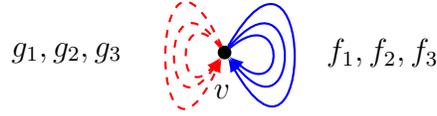
\vspace{-1cm}

\[
\begin{tikzpicture}[scale=0.35]

\node[circle,draw,fill=black,inner sep=0.5pt] (p11) at (1, 1) {$.$} 
edge[-latex, blue, thick,loop, out=40, in=-40, min distance=80, looseness=1.5] (p11)
edge[-latex, blue,thick,loop, out=50, in=-50, min distance=123, looseness=2.2] (p11)
edge[-latex, blue,thick, loop, out=60, in=-60, min distance=190, looseness=2.3] (p11)
edge[-latex, red,thick, loop, dashed, out=130, in=230, min distance=75, looseness=1.5] (p11)
edge[-latex, red,thick, loop, dashed, out=125, in=235, min distance=115, looseness=1.8] (p11)
edge[-latex, red, thick, loop, dashed, out=115, in=245, min distance=190, looseness=1.7] (p11);
                                                               
\node at (0.9, -0.5) {$v$};
\node at (-5, 1) {$g_1,g_2, g_3$};
\node at (7,1) {$f_1,f_2,f_3$};
\end{tikzpicture}
\]
\vspace{-1cm}
\noindent

The factorization rules are given as
\begin{equation*} 
\begin{array}{ll}
g_j f_1= f_2 g_j,  \quad g_j f_2 = f_1 g_j,  \quad g_j f_3 = f_3 g_j & \text{ for } j=1,3 \text{ and } \\
g_2 f_i = f_i g_2  & \text{ for } i=1,2,3.
\end{array}
\end{equation*}
\indent
where $f_i \in \Lambda^{e_1}$ and $g_j \in \Lambda^{e_2}$ for $i,j=1,2,3$.

Then we have the transformations \[v\longrightarrow_2 {\color{red} v}+ {\color{red} v} + {\color{red} v}\longrightarrow_1 {\color{red} v} + {\color{red} v}  +( {\color{blue} v}+ {\color{blue} v} + {\color{blue} v}),\] 

which can be seen pictorially as:
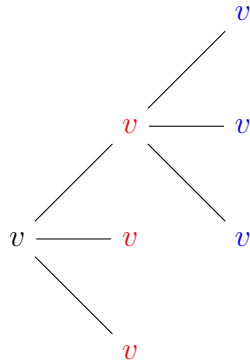
\begin{center}
\begin{tikzpicture}[grow=right]
\node {$v$} 
    child {node {$\color{red} v$}}
    child {node {$ \color{red} v$}}
    child {node {$\color{red} v$}
        child {node {$\color{blue} v$}}
        child {node {$\color{blue} v$}}
        child {node {$\color{blue} v$}}
         } ;
\end{tikzpicture}
\end{center}

Then the $k$-graph monoid is the cyclic monoid   \[M_\Lambda = \frac{\langle v \rangle}{  \langle v = 3v , v = 3v  \rangle} = \{0,v, 2v \}.\]
\end{example}

\subsection{The talented monoid of a $k$-graph}\label{ss introducing the monoid}
Let $\Lambda$ be a row-finite $k$-graph. We now define the talented monoid of $\Lambda$. 
\begin{dfn}\label{def higher rank talented monoid}
The \textit{talented monoid} $T_\Lambda$ of the higher-rank graph $\Lambda$ is the quotient of the free commutative monoid generated by $\{v(n)~|~v\in \Lambda^0, n\in \mathbb{Z}^k\}$ modulo the congruence defined by the following relations:
\begin{equation}\label{TM equation}
    v(n)=\displaystyle{\sum_{\alpha\in v\Lambda^{e_i}}} s(\alpha)(n+e_i)
\end{equation}
for each $v\in \Lambda^0$, $n\in \mathbb{Z}^k$ and $i\in \{1,2,...,k\}$ with $v\Lambda^{e_i}\neq \emptyset$.
\end{dfn}
\begin{rmks}\label{rem connection with talented monoid}
$(i)$ The defining relations (\ref{TM equation}) of $T_\Lambda$ mimic the relations used to define the talented monoid $T_E$ (see Definition \ref{def talented monoid}) of a directed graph $E$ in \cite{Hazrat}. However in (\ref{TM equation}), instead of ranges of paths (as used in $T_E$), we use sources keeping in mind the categorical nature of higher-rank graphs.

$(ii)$ If the $k$-graph $\Lambda$ has no sources then the defining relations (\ref{TM equation}) are equivalent to the more general relations:
\begin{equation}\label{general relation}
    v(n)=\displaystyle{\sum_{\alpha\in v\Lambda^m}} s(\alpha)(n+m)
\end{equation}
for all $v\in \Lambda^0$ and $m\in \mathbb{N}^k$ with $v\Lambda^m\neq \emptyset$.
\end{rmks}

For $v(n)\in T_\Lambda$, the $n$ in the parenthesis is called the \textit{state} of $v(n)$. We often denote $v(0)$ simply as $v$. The monoid $T_\Lambda$ is equipped with a natural $\mathbb{Z}^k$-monoid structure with respect to the action of $\mathbb{Z}^k$ on generators given by state shift as follows:

\begin{center}
$^n v(m):=v(m+n)$ for all $n\in \mathbb{Z}^k$
\end{center}
and then extended linearly. This action is clearly compatible with the algebraic pre-ordering on $T_\Lambda$. This also distinguishes $T_\Lambda$ from the $k$-graph monoid $M_\Lambda$ and throughout the rest of the paper this action is of utmost interest to us.

To understand the effect of the relations (\ref{general relation}), we compute the talented monoid of a simple looking $2$-graph in the following example.
\begin{example}\label{ex talented monoid of f^*(C_n)}
Suppose $C_4$ is the cycle of length $4$ and $\Lambda:=f^*(C_4^*)$, the $2$-graph constructed using the technique of Example \ref{ex higher-rank graphs} $(v)$. The following figure represents the $1$-skeleton of $\Lambda$.

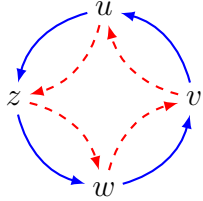
\[
\begin{tikzpicture}[scale=1.2]
\node[inner sep=1.5pt, circle] (A) at (0,1) {$u$};	
\node[inner sep=1.5pt, circle] (B) at (-1,0) {$z$};	
\node[inner sep=1.5pt, circle] (C) at (0,-1) {$w$};
\node[inner sep=1.5pt, circle] (D) at (1,0) {$v$};	
	
\path[->, red, dashed, >=latex,thick] (A) edge [bend left] node[]{} (B);
\path[->, red, dashed, >=latex,thick] (B) edge [bend left] node[]{} (C);
\path[->, red, dashed, >=latex,thick] (C) edge [bend left] node[]{} (D);
\path[->, red, dashed, >=latex,thick] (D) edge [bend left] node[]{} (A);

\path[->,blue, >=latex,thick] (A) edge [bend right] node[]{} (B);
\path[->,blue, >=latex,thick] (B) edge [bend right] node[]{} (C);
\path[->,blue, >=latex,thick] (C) edge [bend right] node[]{} (D);
\path[->,blue, >=latex,thick] (D) edge [bend right] node[]{} (A);

\end{tikzpicture}
\]
The solid blue edges have degree $(1,0)$ while the dashed red edges are of degree $(0,1)$. Using the defining relations of $T_\Lambda$ we have,
\begin{align*}
    z &=u(e_1)=u(e_2)\\
    w &=u(2e_1)=u(e_1+e_2)=u(2e_2)\\
    v &=u(3e_1)=u(2e_1+e_2)=u(e_1+2e_2)=u(3e_2)\\
    u &=u(4e_1)=u(3e_1+e_2)=u(2e_1+2e_2)=u(e_1+3e_2)=u(4e_2).\\
\end{align*}
It follows that the vertices $v,w,z$ all lie in the $\mathbb{Z}^2$-orbit of $u$. A simple computation yields that for any $n=(n_1,n_2)\in \mathbb{Z}^2$, $u(n)=u(xe_1+ye_2)$ where $x,y\in \{0,1,2,3\}$ are such that $n_1\equiv x$ (mod $4$) and $n_2\equiv y$ (mod $4$). Using the above relations we further write $u(n)=u(Ne_1)$ where $N\in \{0,1,2,3\}$ and $(x+y)\equiv N$ (mod $4$). The talented monoid $T_\Lambda$ is therefore the commutative monoid freely generated by $u,u(e_1),u(2e_1)$ and $u(3e_1)$ or equivalently $T_\Lambda\cong \mathbb{N}\oplus \mathbb{N}\oplus \mathbb{N}\oplus \mathbb{N}$, where the $\mathbb{Z}^2$-action on the later is given by cyclic permutation more precisely, for any $n=(n_1,n_2)\in \mathbb{Z}^2$ and $(a,b,c,d)\in \mathbb{N}\oplus \mathbb{N}\oplus \mathbb{N}\oplus \mathbb{N}$, $^n (a,b,c,d)$ is the $4$-tuple obtained by permuting the entries $N$-times cyclically where $(n_1+n_2) \equiv N$ (mod $4$). It follows that $T_\Lambda$ and $T_{C_4}$ (see \cite[Example 2.4]{Hazrat}) are isomorphic as commutative monoids.   
\end{example}

\begin{rmk}\label{rem motivation towards aperiocity characterization}
In the $2$-graph $\Lambda$ of Example \ref{ex talented monoid of f^*(C_n)}, every infinite path $x$ is periodic with period $(1,-1)$ and so $\Lambda$ is not aperiodic (see Definition \ref{def KP convention}). On the other hand $\mathbb{Z}^2$ does not act freely on $T_\Lambda$ in fact $^{(1,0)}a=~^{(0,1)}a$ for any $a\in T_\Lambda$. This connection is not a mere coincidence as we prove in Theorem \ref{th Equivalence for aperiodicity in UF} of Section \ref{sec aperiodicity and KP algebra}, that when $T_\Lambda$ is atomic, then $\Lambda$ is aperiodic if and only if every element $a\in T_\Lambda$ has trivial isotropy under the $\mathbb{Z}^k$-action.
\end{rmk}

As the path category of a directed graph is simply a $1$-graph, the following connection is quite expected.
\begin{prop}\label{pro the obvious connection}
Let $E$ be a row-finite directed graph and $E^*$ the path category of $E$, considered as a $1$-graph. Then $T_{E^*}\cong T_E$.
\end{prop}
\begin{proof}
As the defining relations of $T_{E^*}$ are modified in accordance with the reversal of edge direction in the passage from $E$ to $E^*$, it is straightforward to see that the canonical map $v(i)\longmapsto v(i)$; $v\in E^0$, $i\in \mathbb{Z}$ when extended linearly gives a well-defined isomorphism of monoids between $T_E$ and $T_{E^*}$.
\end{proof}

For $k$-graph $\Lambda$ without sources, the monoid $T_\Lambda$ has appeared in the literature before but not in its present form. In \cite{Ara-Hazrat}, the authors considered a directed system $(D_n,~\phi_{m,n})$ where for each $n\in \mathbb{Z}^k$, $D_n:=\mathbb{N}\Lambda^0$ (the free commutative monoid generated by $\Lambda^0$) and for all $m,n\in \mathbb{Z}^k$ with $m\le n$ the connecting morphism $\phi_{m,n}:D_m\longrightarrow D_n$ is defined on the generators by 
\begin{center}
    $\phi_{m,n}(v):=\displaystyle{\sum_{w\in \Lambda^0}} |v\Lambda^{n-m}w| w=\displaystyle{\sum_{\lambda\in v\Lambda^{n-m}}} s(\lambda)$ for all $v\in \Lambda^0$.
\end{center}
The direct limit of the above directed system was shown to be isomorphic with the graded $\mathcal{V}$-monoid of the Kumjian-Pask algebra $\KP_\mathsf{k}(\Lambda)$ (see \cite[Corollary 6.7]{Ara-Hazrat}). We obtain the following proposition which says that $T_\Lambda$ is disguised as the directed limit of the above directed system.
\begin{prop}\label{pro realizing as a direct limit}
Let $\Lambda$ be a row-finite $k$-graph without sources. Then 

$(i)$ $T_\Lambda= \displaystyle{\varinjlim_{\mathbb{Z}^k}} (D_n,\phi_{m,n})$;

$(ii)$ the monoid $T_\Lambda$ is cancellative. 
\end{prop}
\begin{proof}
$(i)$ For each $n\in \mathbb{Z}^k$, we denote the generators in $D_n=\mathbb{N}\Lambda^0$ as $v_n$, $v\in \Lambda^0$. Define $\psi_n:D_n \longrightarrow T_\Lambda$ by $v_n \longmapsto \psi_n(v_n):=v(n)$ for all $v\in \Lambda^0$ and then extend linearly to all elements of $D_n$. Clearly $\psi_n$ induces a monoid homomorphism for each $n\in \mathbb{Z}^k$.

Let $m\le n\in \mathbb{Z}^k$. We claim that $\psi_n \circ \phi_{m,n}=\psi_m$. Choose any $v\in \Lambda^0$. Then
\begin{center}
$\phi_{m,n}(v_m)=\displaystyle{\sum_{\lambda\in v\Lambda^{n-m}}} s(\lambda)_n$,    
\end{center}
and so \[\psi_n(\phi_{m,n}(v_m))=\displaystyle{\sum_{\lambda\in v\Lambda^{n-m}}} \psi_n(s(\lambda)_n)=\displaystyle{\sum_{\lambda\in v\Lambda^{n-m}}} s(\lambda)(n)=\displaystyle{\sum_{v\in \Lambda^{n-m}}}s(\lambda)(m+(n-m))=v(m)=\psi_m(v_m).\]
Hence the claim is established. Now we just need to prove that $(T_\Lambda,\psi_n)_{n\in \mathbb{Z}^k}$ satisfies the universal property required to be a direct limit. So assume that $M$ is any monoid and $\chi_n:D_n\longrightarrow M$ are monoid homomorphisms satisfying $\chi_n\circ \phi_{m,n}=\chi_m$ for all $m\le n\in \mathbb{Z}^k$. We define $\chi:T_\Lambda \longrightarrow M$ by $v(i) \longmapsto \chi(v(i)):=\chi_i(v_i)$ for all $v\in \Lambda^0$ and $i\in \mathbb{Z}^k$. In order to show that $\chi$ induces a well-defined monoid homomorphism, choose any $v\in \Lambda^0$, $n\in \mathbb{Z}^k$ and $m\in \mathbb{N}^k$. In $T_\Lambda$ we have $v(n)=\displaystyle{\sum_{\alpha\in v\Lambda^m}}s(\alpha)(n+m)$. We need to show that $\chi$ is compatible with this relation. Now
\begin{align*}
    \chi(v(n))=\chi_n(v_n)=\chi_{n+m}(\phi_{n,n+m}(v_n)) &=\chi_{n+m}\left(\displaystyle{\sum_{\alpha\in v\Lambda^m}} s(\alpha)_{n+m}\right)\\
    &= \displaystyle{\sum_{\alpha\in v\Lambda^m}} \chi_{n+m}(s(\alpha)_{n+m})\\
    &= \displaystyle{\sum_{\alpha\in v\Lambda^m}} \chi(s(\alpha)(n+m))\\
    &=\chi\left(\displaystyle{\sum_{\alpha\in v\Lambda^m}} s(\alpha)(n+m)\right).
\end{align*}
Therefore $\chi$ is well-defined. Now from the definition of $\chi$ it is clear that $\chi$ is the unique monoid homomorphism satisfying $\chi\circ \psi_n=\chi_n$ for all $n\in \mathbb{Z}^k$. This completes the proof of $(i)$.

$(ii)$ This is immediate since $\displaystyle{\varinjlim_{\mathbb{Z}^k}} (D_n,\phi_{m,n})\cong \mathcal{V}^{\gr}(\KP_\mathsf{k}(\Lambda))$ and $\mathcal{V}^{\gr}(\KP_\mathsf{k}(\Lambda))$ is cancellative by \cite[Corollary 6.6]{Ara-Hazrat}.
\end{proof}

\begin{rmk}\label{rem Group completion of talented monoid}
From Proposition \ref{pro realizing as a direct limit} and \cite[Corollary 6.7]{Ara-Hazrat}, one can conclude that the group completion of the monoid $T_\Lambda$ is the graded Grothendieck group $K_0^{\gr}(\KP_\mathsf{k}(\Lambda))$. 
\end{rmk}

Recall the skew-product $k$-graph $\overline{\Lambda}:=\Lambda\times_d \mathbb{Z}^k$ from \cite{Kumjian-Pask}. There is a canonical action of $\ZZ^k$ on $M_{\overline{\Lambda}}$ defined by ${}^n (v,i) \longmapsto (v,(i+n))$ on the generators and extended linearly on the whole monoid.  

The following is a generalisation of \cite[Lemma 3.2]{Hazrat}.

\begin{thm} \label{thm:talisrefine}
Let $\Lambda$ be a row-finite $k$-graph without sources. Then 
\begin{enumerate}[\upshape(1)]
\item[(i)] there is a forgetful homomorphism of monoids $\pi:T_\Lambda \longrightarrow M_\Lambda$ given by $v(i) \longmapsto v$. Furthermore, we have 
\[\frac{T_\Lambda}{\langle v = v(n), v\in \Lambda^0, n \in \mathbb{Z}^k \rangle} \cong M_\Lambda;\] 

\item[(ii)] the map $\phi: T_\Lambda \longrightarrow M_{\overline{\Lambda}}$ defined by $v(n) \longmapsto (v,n)$ induces a $\ZZ^k$-monoid isomorphism.

\end{enumerate}
\end{thm}

\begin{proof}
$(i)$ The fact that $\pi:T_\Lambda \longrightarrow M_\Lambda, v(i) \longmapsto v$ is well-defined, is immediate. Define a binary relation $\curly_1$ on $T_\Lambda$ as follows:

\begin{equation}\label{ihesdin}
\sum_{l=1}^q v_l (n_l) \curly_1 \sum_{l \neq p} v_l(n_l) +   v_p(n_p+n), \text{ for some } n\in \mathbb N^k.
\end{equation}

Let $\curly$ be the transitive and reflective closure of $\curly_1$ on $T_\Lambda$, that is $\alpha \curly \beta$ if and only if either $\alpha=\beta$ or there is a finite string $\alpha = \alpha_0 \curly_1 \alpha_1 \curly_1 \cdots \curly_1 \alpha_l = \beta$. Let $\approx$ be the congruence on $T_\Lambda$ generated by $\curly$. This congruence relation respects the action of $\mathbb Z^k$. We leave it to the reader to show that $\approx$ is in fact the congruence specified in the statement and $T_E/\approx$ is isomorphic to $M_\Lambda$. 

$(ii)$ 
We check that $\phi$ is well-defined. In $T_\Lambda$ we have
$v(n) = \displaystyle{\sum_{\lambda \in  v \Lambda^m}} s(\lambda)(n +m)$ for any $m\in \mathbb{N}^k$.
Then 
\begin{align*}
\phi ( v(n)) =(v,n) & = \displaystyle{\sum_{(\lambda,n)\in (v,n)\overline{\Lambda}^m}} \overline{s}(\lambda,n) ~\text{(using Proposition \ref{pro another realization of k-graph monoid})}\\
&=  \displaystyle{\sum_{\lambda\in v\Lambda^m}} (s(\lambda),n+m) \\
&=\phi \left(\displaystyle{\sum_{\lambda\in v\Lambda^m}} s(\lambda)(n+m)\right).
\end{align*}
To establish $\phi$ is a  $\ZZ^k$-monoid isomorphism is now straightforward. 
\end{proof}

Throughout this article we conveniently switch between $T_\Lambda$ and  $M_{\overline{\Lambda}}$. In particular Theorem~\ref{thm:talisrefine} allows us to use all the results proved for the monoid $M_\Lambda$, such as Lemma \ref{lem the splitting lemma} and confluence Lemma \ref{lem the confluence lemma} for $T_\Lambda$. 

We conclude this section by proving a nice connection between $T_\Lambda$ and the type monoid of the inverse semigroup of the compact open bisections of the skew-product groupoid $\mathcal{G}_\Lambda\times_c \mathbb{Z}^k$ where $c:\mathcal{G}_\Lambda\longrightarrow \mathbb{Z}^k$ is the canonical cocycle. This will directly prove that $T_\Lambda$ is indeed a refinement monoid (see \cite[Corollary 4.1.4]{Wehrung}).

Let $\Lambda$ be a row-finite $k$-graph with no sources. We wish to show that the talented monoid $T_\Lambda$ is isomorphic (as $\mathbb{Z}^k$-monoids) with the graded Type semigroup $\Typ^{\gr}(\mathcal{G}_\Lambda)$. Consider the correspondence 
\begin{align*}
\varphi: T_\Lambda &\longrightarrow \Typ^{\gr}(\mathcal{G}_\Lambda)=\Typ(\mathcal{G}_\Lambda\times_c \mathbb{Z}^k)\\
v(n) &\longmapsto \typ(Z(v)\times \{n\}).
\end{align*}
We show that the above correspondence when extended linearly gives rise to a well-defined monoid homomorphism.

Let $v\in \Lambda^0$ and $n\in \mathbb{Z}^k$. Then for each $m\in \mathbb{N}^k$, we have 
\begin{center}
    $v(n)=\displaystyle{\sum_{\lambda\in v\Lambda^m}} s(\lambda)(n+m)$.
\end{center}
We have to show that $\varphi$ respects the above relation. Let $\lambda\in v\Lambda^m$ and $x:=Z(\lambda,s(\lambda))\times \{n\}\in (\mathcal{G}_\Lambda \times_{c} \mathbb{Z}^k)^a$. Using the Definition of skew-product groupoid (see \cite[Definition 1.6]{Renault}), we have
\[x^{-1}x =(Z(\lambda,s(\lambda))\times \{n\})^{-1}(Z(\lambda,s(\lambda))\times \{n\})=(Z(s(\lambda),\lambda)\times \{n+m\})(Z(\lambda,s(\lambda))\times \{n\})=Z(s(\lambda))\times \{n+m\}\]
and 
\[xx^{-1}=(Z(\lambda,s(\lambda))\times \{n\})(Z(s(\lambda),\lambda)\times \{n+m\})=Z(\lambda)\times \{n\}.\]
Therefore $\typ(Z(s(\lambda))\times \{n+m\})=\typ(Z(\lambda)\times \{n\})$ using \cite[Definition 3.1 (2)]{Cordeiro}. 

Now $Z(v)=\displaystyle{\bigsqcup_{\lambda\in v\Lambda^m}} Z(\lambda)$ and so 
\begin{align*}
\varphi(v(n))= \typ(Z(v)\times \{n\}) &= \typ\left(\displaystyle{\bigsqcup_{\lambda\in v\Lambda^m}}(Z(\lambda)\times \{n\})\right)\\
&= \displaystyle{\sum_{\lambda\in v\Lambda^m}} \typ(Z(\lambda)\times\{n\})\\
&=\displaystyle{\sum_{\lambda\in v\Lambda^m}} \typ(Z(s(\lambda))\times \{n+m\})\\
&=\displaystyle{\sum_{\lambda\in v\Lambda^m}} \varphi(s(\lambda)(n+m))\\
&=\varphi\left(\displaystyle{\sum_{\lambda\in v\Lambda^m}} s(\lambda)(n+m)\right).
\end{align*}

Thus $\varphi$ is a well-defined $\mathbb{Z}^k$-monoid homomorphism. 

We need one lemma before we can prove that the above defined homomorphism is in fact bijective.

\begin{lem}\label{lem relative complement and disjoint union}
Let $\Lambda$ be a row-finite $k$-graph without sources. Let $\mathcal{B}:=\{Z(\lambda,\mu)~|~\lambda,\mu\in \Lambda; s(\lambda)=s(\mu)\}$ be the basis of the topology of $\mathcal{G}_\Lambda$ as described in \cite{Kumjian-Pask}. Then

$(i)$ $Z(\lambda,\mu)\cap Z(\gamma,\delta)=\displaystyle{\bigsqcup_{(\alpha,\beta)\in F}} Z(\lambda\alpha,\delta\beta)$ where $F:=\Lambda^{\min}(\lambda,\gamma)\cap \Lambda^{\min}(\mu,\delta)$;

$(ii)$ for each $B_1,B_2\in \mathcal{B}$, $B_1\setminus B_2$ is a finite disjoint union of sets from $\mathcal{B}$;

$(iii)$ for each $\lambda,\mu\in \Lambda$, $Z(\lambda)\setminus Z(\mu)=\displaystyle{\bigsqcup_{\substack{\alpha\in s(\lambda)\Lambda^m\\ \alpha\notin Ext(\lambda;\{\mu\}) }}} Z(\lambda\alpha)$ where $m:=(d(\lambda)\vee d(\mu))-d(\lambda)$.
\end{lem}
\begin{proof}
$(i)$ Follows from \cite[Lemma 5.5]{Clark}.

$(ii)$ Let $B_1=Z(\lambda,\mu)$ and $B_2=Z(\gamma,\delta)$. Then 
\begin{align*}
    B_1\setminus B_2 & =Z(\lambda,\mu)\setminus Z(\gamma,\delta)\\
    & =Z(\lambda,\mu)\setminus \displaystyle{\bigsqcup_{(\alpha,\beta)\in F}} Z(\lambda\alpha,\delta\beta) ~\text{(from (i))}\\
    &=Z(\lambda,\mu)\setminus \displaystyle{\bigsqcup_{\alpha\in D(F)}} Z(\lambda\alpha,\mu\alpha), ~\text{where}~ D(F):=\{\alpha\in \Lambda~|~(\alpha,\beta)\in F ~\text{for some}~ \beta\in \Lambda\}\\
    &=\displaystyle{\bigsqcup_{\eta\in s(\lambda)\Lambda^n ~\setminus ~D(F)}} Z(\lambda\eta,\mu\eta), ~\text{where}~ n:=(d(\lambda)\vee d(\gamma))-d(\lambda)=(d(\mu)\vee d(\delta))-d(\mu).
\end{align*}
$(iii)$ Follows from $(i)$ and $(ii)$ since $Z(\lambda)=Z(\lambda,\lambda)$.
\end{proof}

Now we are in a position to prove that $\varphi:T_\Lambda\longrightarrow \Typ^{\gr}(\mathcal{G}_\Lambda)$ is bijective. Let $U$ be any compact open subset of $\mathcal{G}_\Lambda^{(0)}$ and $n\in \mathbb{Z}^k$. Using Lemma \ref{lem relative complement and disjoint union} $(i)$ and $(iii)$, one can write $U\times \{n\}=\displaystyle{\bigsqcup_{i=1}^{t}} Z(\lambda_i)\times \{n\}$ where $t\in \mathbb{N}$. Then $\typ(U\times \{n\})=\displaystyle{\sum_{i=1}^{t}} \typ(Z(\lambda_i)\times \{n\})=\displaystyle{\sum_{i=1}^{t}}\typ(Z(s(\lambda_i))\times \{n+d(\lambda_i)\})=\varphi\left(\displaystyle{\sum_{i=1}^{t}} s(\lambda_i)(n+d(\lambda_i))\right)$. Therefore $\varphi$ is surjective. Now we show that $\varphi$ is injective as well.

Let $a=\displaystyle{\sum_{i=1}^{p}} v_i(n_i)$ and $b=\displaystyle{\sum_{j=1}^{q}} w_j(m_j)\in T_\Lambda$ be such that $\varphi(a)=\varphi(b)$. So $\displaystyle{\sum_{i=1}^{p}} \typ(Z(v_i)\times n_i)=\displaystyle{\sum_{j=1}^{q}} \typ(Z(w_j)\times m_j)$ in $\Typ^{\gr}(\mathcal{G}_\Lambda)$. Using the isomorphism of \cite[Proposition 7.3]{Bosa} we can argue that in $S(\mathcal{G}_\Lambda\times_c \mathbb{Z}^k, (\mathcal{G}_\Lambda\times_c \mathbb{Z}^k)^a)$,
\begin{center}
    $A:=\displaystyle{\bigcup_{i=1}^{p}} (Z(v_i)\times n_i)\times \{i\}\sim B:=\displaystyle{\bigcup_{j=1}^{q}} (Z(w_j)\times m_j)\times \{j\}$.
\end{center}
Now by the definition of $\sim$ and using Lemma \ref{lem relative complement and disjoint union} $(i)$, $(ii)$ we can find $l\in \mathbb{N}$, $\lambda_1,\lambda_2,...,\lambda_l,\mu_1,\mu_2,...,\mu_l\in \Lambda$; $t_1,t_2,...,t_l\in \mathbb{Z}^k$ and $p_1,p_2,...,p_l,q_1,q_2,...,q_l\in \mathbb{N}$ such that
\begin{center}
    $A=\displaystyle{\bigsqcup_{k=1}^{l}}s(Z(\lambda_k,\mu_k)\times t_k)\times \{p_k\}$ and $B=\displaystyle{\bigsqcup_{k=1}^{l}}r(Z(\lambda_k,\mu_k)\times t_k)\times \{q_k\}$.
\end{center}
It implies that one can partition $\{1,2,...,l\}$ as $\displaystyle{\sqcup_{i=1}^{p}} I_i$ and $\displaystyle{\sqcup_{j=1}^{q}} J_j$ such that 
\begin{center}
    $Z(v_i)\times n_i=\displaystyle{\bigsqcup_{k\in I_i}} Z(\mu_k)\times \{t_k+d(\lambda_k)-d(\mu_k)\}$ and $Z(w_j)\times m_j=\displaystyle{\bigsqcup_{k\in J_j}} Z(\lambda_k)\times \{t_k\}$
\end{center}
for all $i=1,2,...,p$ and $j=1,2,...,q$.

Therefore $m_j=t_k$ for all $k\in J_j$ and $n_i+d(\mu_k)=t_k+d(\lambda_k)$ for all $k\in I_i$. Moreover $Z(v_i)=\displaystyle{\bigsqcup_{k\in I_i}} Z(\mu_k)$ and $Z(w_j)=\displaystyle{\bigsqcup_{k\in J_j}} Z(\lambda_k)$. 

Now we claim that $(v_i,0)=\displaystyle{\sum_{k\in I_i}} (s(\mu_k),d(\mu_k))$ in $M_{\overline{\Lambda}}$. Let $n=\displaystyle{\bigvee_{k\in I_i}} d(\mu_k)$. For each $\alpha\in v_i\Lambda^n$, $Z(\alpha)\subseteq Z(v_i)=\displaystyle{\bigsqcup_{k\in I_i}} Z(\mu_k)$ and so there exist a unique $k'\in I_i$ with the property that $\alpha(0,d(\mu_k'))=\mu_k'$ otherwise it will contradict the disjoint union decomposition of $Z(v_i)$. This argument shows that
\begin{center}
    $\displaystyle{\sum_{k\in I_i}} (s(\mu_k),d(\mu_k))\longrightarrow \displaystyle{\sum_{k\in I_i}}~\left(\displaystyle{\sum_{\xi\in s(\mu_k)\Lambda^{n-d(\mu_k)}}} (s(\xi),n)\right)=\displaystyle{\sum_{\alpha\in v_i\Lambda^n}} (s(\alpha),n)$.
\end{center}
in $\mathbb{F}_{_{\overline{\Lambda}}}$. Now our claim is established by using Lemma \ref{lem two sum equivalent} and Lemma \ref{lem the confluence lemma}. Therefore in $T_\Lambda$, we have $v_i=\displaystyle{\sum_{k\in I_i}} s(\mu_k)(d(\mu_k))$. Similarly $w_j=\displaystyle{\sum_{k\in J_j}} s(\lambda_k)(d(\lambda_k))$. Finally we have
\begin{align*}
a=\displaystyle{\sum_{i=1}^{p}} v_i(n_i)& =\displaystyle{\sum_{i=1}^{p}} \left(\displaystyle{\sum_{k\in I_i}} s(\mu_k)(n_i+d(\mu_k))\right)\\
& =\displaystyle{\sum_{k=1}^{l}}s(\mu_k)(t_k+d(\lambda_k))\\
&=\displaystyle{\sum_{k=1}^{l}} s(\lambda_k)(t_k+d(\lambda_k))\\
&=\displaystyle{\sum_{j=1}^{q}}\left(\displaystyle{\sum_{k\in J_j}} s(\lambda_k)(m_j+d(\lambda_k))\right)\\
&=\displaystyle{\sum_{j=1}^{q}} w_j(m_j)=b. 
\end{align*}
Thus $\varphi$ is a $\mathbb{Z}^k$-monoid isomorphism and we have proved the following: 

\begin{thm}\label{th isomorphism between talented monoid and type monoid}
For any row-finite $k$-graph without sources $\Lambda$, the talented monoid $T_\Lambda\cong T_{\mathcal{G}_\Lambda}$, and hence $T_\Lambda$ is a conical refinement monoid.    
\end{thm}

\begin{rmks}\label{rem importance of the isomorphism}
$(i)$ It is well known that the fundamental properties of a $k$-graph $\Lambda$ are intrinsically connected to the properties of the path groupoid $\mathcal{G}_\Lambda$ (see \cite[Propositions 6.3, 7.1]{Clark}). In the following sections we show that the properties like cofinality, aperiodicity, and strong aperiodicity of $\Lambda$ are nicely captured by the talented monoid $T_\Lambda$. Thus in view of Theorem \ref{th isomorphism between talented monoid and type monoid} we get strong evidence, in the level of $k$-graphs, that the graded type semigroup of a $\Gamma$-graded ample groupoid with its $\Gamma$-action encodes important information about the groupoid and its associated Steinberg algebra which is in line with the Program \ref{genproblem} in the introduction.

$(ii)$ Following a similar line of argument (with a little bit of required modifications) used in the above proof, one can also obtain the non-graded version of Theorem \ref{th isomorphism between talented monoid and type monoid}, namely $M_\Lambda\cong \Typ(\mathcal{G}_\Lambda)$. Via this isomorphism we can relate our $k$-graph monoid to the semigroup $S(\Lambda)$ defined in \cite[Definition 3.5]{PSS}. Since for any ample groupoid $\mathcal{G}$, the type semigroups $\Typ(\mathcal{G})$, $S(\mathcal{G},\mathcal{G}^a)$ (B\"{o}nicke and Li \cite{Bonicke}), and $S(\mathcal{G})$ (Rainone and Sims \cite{Rainone}) are all isomorphic (due to \cite[Proposition 7.3]{Bosa} and \cite[Remark 5.5]{Rainone}), we infer that $M_\Lambda\cong S(\mathcal{G}_\Lambda)\cong S(\Lambda)$ where the last isomorphism is obtained in \cite[Proposition 8.11]{Rainone}. Also in view of Theorem \ref{thm:talisrefine} $(ii)$, $T_\Lambda\cong S(\overline{\Lambda})$. Since $S(\Lambda)$ is the quotient of the free commutative monoid $\mathbb{N}\Lambda^0$ by a certain congruence $\approx_\Lambda$ defined using the matrices $A_n\in \mathbb{M}_{\Lambda^0}(\mathbb{N})$, $n\in \mathbb{N}^k$ where  $A_n(v,w):=|v\Lambda^n w|$ for all $v,w\in \Lambda^0$ (see \cite{PSS, Rainone} for more details), we believe that these isomorphisms may help to connect the concepts of symbolic dynamics such as shift equivalence and strong shift equivalence with the monoids $M_\Lambda$ and $T_\Lambda$.
\end{rmks}

\section{Cofinality of a $k$-graph via $T_\Lambda$}\label{sec cofinality}
In this section our initial aim is to set up a lattice isomorphism between the lattice of all hereditary saturated subsets of a $k$-graph $\Lambda$ and the lattice of all $\mathbb{Z}^k$-order ideals of $T_\Lambda$. This will help us to find a necessary and sufficient condition for cofinality of $\Lambda$ using the monoid $T_\Lambda$. Throughout the rest of this section we will assume that $\Lambda$ is a row-finite $k$-graph without sources.

Let $J$ be a $\mathbb{Z}^k$-order ideal of $T_\Lambda$. We define 
\begin{center}
$H_{_J}:=\{v\in \Lambda^0~|~v(0)\in J\}$. 
\end{center}
\begin{lem}\label{lem HS subset from order ideal}
Let $J$ be a $\mathbb{Z}^k$-order ideal of $T_\Lambda$. Then $H_{_J}$ is a hereditary and saturated subset of $\Lambda^0$.
\end{lem}
\begin{proof}
Let $\lambda\in \Lambda$ be such that $r(\lambda)\in H_{_J}$. Then $s(\lambda)(d(\lambda))\le \displaystyle{\sum_{\mu\in r(\lambda)\Lambda^{d(\lambda)}}} s(\mu)(d(\lambda))=r(\lambda)$ in $T_\Lambda$. Since $J$ is a $\mathbb{Z}^k$-order ideal, it implies that $s(\lambda)(0)=~^{-d(\lambda)} s(\lambda)(d(\lambda))\in J$ and so $s(\lambda)\in H_{_J}$. Thus $H_{_J}$ is hereditary. Let $v\in \Lambda^0$ and $n\in \mathbb{N}^k$ such that $s(v\Lambda^n)\subseteq H_{_J}$. Then it follows that $v(0)=\displaystyle{\sum_{\lambda\in v\Lambda^n}} s(\lambda)(n)\in J$. Consequently $v\in H_{_J}$. Therefore $H_{_J}$ is saturated as well.
\end{proof}

For a subset $X$ of $\Lambda^0$, by $\langle X\rangle$ we denote the $\mathbb{Z}^k$-order ideal of $T_\Lambda$ generated by $\{v(0)~|~v\in X\}$ if $X\neq \emptyset$ and $\langle X \rangle =\{0\}$ if $X=\emptyset$. For $X(\neq \emptyset)\subseteq \Lambda^0$, it is not hard to show that
\begin{center}
    $\langle X\rangle=\{x\in T_\Lambda~|~x\le \displaystyle{\sum_{i=1}^{N}}~^{n_i}v_i$ for some $N\in \mathbb{N};~n_i\in \mathbb{Z}^k$ and $v_i\in X$ for all $i=1,2,...,N\}$.
\end{center}
\begin{lem}\label{lem order ideal generated by HS subset}
Let $J$ be any $\mathbb{Z}^k$-order ideal of $T_\Lambda$. Then $J=\langle H_{_J}\rangle$. 
\end{lem}
\begin{proof}
The inclusion $\langle H_{_J} \rangle \subseteq J$ is quite clear. For the other inclusion let $0\neq x\in J$. Suppose $x=\displaystyle{\sum_{i=1}^{t}} v_i(n_i)$ is any representation of $x$ in $T_\Lambda$. Then for each $i=1,2,...,t$ we have $v_i(n_i)\le x$ and so $v_i=~^{-n_i} v_i(n_i)\in J$ since $J$ is a $\mathbb{Z}^k$-order ideal. This implies $v_i\in H_{_J}$ for each $i=1,2,...,t$ and hence $x=\displaystyle{\sum_{i=1}^{t}} v_i(n_i)\in \langle H_{_J} \rangle$. 
\end{proof}
We obtain the following proposition which gives a necessary condition for cofinality of $\Lambda$.
\begin{prop}\label{pro necessary condition for cofinality}
Let $\Lambda$ be a row-finite $k$-graph without sources. If $\Lambda$ is cofinal then $T_\Lambda$ is a simple $\mathbb{Z}^k$-monoid.    
\end{prop}
\begin{proof}
By \cite[Lemma 5.12]{Pino} the only hereditary and saturated subsets of $\Lambda^0$ are $\emptyset$ and $\Lambda^0$ itself. Suppose $J$ is any $\mathbb{Z}^k$-order ideal of $T_\Lambda$. Then $H_{_J}$ is a hereditary and saturated subset by Lemma \ref{lem HS subset from order ideal} and $J=\langle H_{_J}\rangle$ by Lemma \ref{lem order ideal generated by HS subset}. Now if $H_{_J}=\emptyset$ then $J=\langle \emptyset \rangle=\{0\}$ and if $H_{_J}=\Lambda^0$ then $J=\langle \Lambda^0\rangle =T_\Lambda$.
\end{proof}
The next lemma helps us to prove our desired lattice isomorphism as well as the converse of Proposition \ref{pro necessary condition for cofinality}.
\begin{lem}\label{lem recovering HS subset}
Let $H$ be any hereditary and saturated subset of $\Lambda^0$. Let $J_{_H}:=\langle H\rangle$. Then $H=H_{J_{_H}}$.   
\end{lem}
\begin{proof}
The inclusion $H\subseteq H_{J_{_H}}$ is clear. Take $v\in H_{J_{_{H}}}$. Then $v\in \langle H\rangle$. Hence $v\le \displaystyle{\sum_{i=1}^{t}} u_i(n_i)$ for some $t\in \mathbb{N}$, $u_i\in H$ and $n_i\in \mathbb{Z}^k$ for all $i=1,2,..,t$. It implies $v+x=\displaystyle{\sum_{i=1}^{t}} u_i(n_i)$ for some $x\in T_\Lambda$. Now we use the isomorphism $\phi:T_\Lambda\longrightarrow M_{\overline{\Lambda}}$ to bypass this equality to $M_{\overline{\Lambda}}$. Let $y=\phi(x)$. Then we have 
\begin{center}
    $(v,0)+y\sim \displaystyle{\sum_{i=1}^{t}} (u_i,n_i)$ in $\mathbb{F}_{_{\overline{\Lambda}}} \setminus \{0\}$.
\end{center}
By Lemma \ref{lem the confluence lemma}, there exists $\gamma\in \mathbb{F}_{_{\overline{\Lambda}}} \setminus \{0\}$ such that $(v,0)+y\longrightarrow \gamma$ and $\displaystyle{\sum_{i=1}^{t}} (u_i,n_i)\longrightarrow \gamma$. Since $H$ is hereditary so is $H\times \mathbb{Z}^k$ in $\overline{\Lambda}^0$. Therefore each vertex in the basic representation of $\gamma$ in $\mathbb{F}_{_{\overline{\Lambda}}}$ is in $H\times \mathbb{Z}^k$. Now by Lemma \ref{lem the splitting lemma}, there are $\gamma_1,\gamma_2\in \mathbb{F}_{_{\overline{\Lambda}}} \setminus \{0\}$ such that $(v,0)\longrightarrow \gamma_1$, $y\longrightarrow\gamma_2$ and $\gamma=\gamma_1+\gamma_2$. Then we have a chain of arrows:
\begin{center}
    $(v,0)\longrightarrow_{j_1}\eta_1\longrightarrow_{j_2}\eta_2\longrightarrow_{j_3}\cdot\cdot\cdot\longrightarrow_{j_m}\eta_m=\gamma_1$. 
\end{center}
Since $\eta_{m-1}\longrightarrow_{j_m}\gamma_1$, so $\eta_{m-1}=\eta_{m-1}'+(w,l)$ and $\gamma_1=\eta_{m-1}'+\mathfrak{s}^{j_m}((w,l))$ for some $w\in \Lambda^0$ and $l\in \mathbb{Z}^k$. 
As $\overline{s}((w,l)\overline{\Lambda}^{e_{j_m}})\subseteq H\times \mathbb{Z}^k$ 
and $H\times \mathbb{Z}^k$ is saturated so $(w,l)\in H\times \mathbb{Z}^k$. So it follows that $\supp(\eta_{m-1})\subseteq H\times \mathbb{Z}^k$. Proceeding in this way, we can show that $\supp(\eta_1)\subseteq H\times \mathbb{Z}^k$ which in turn implies $\overline{s}((v,0)\overline{\Lambda}^{e_{j_1}})\subseteq H\times \mathbb{Z}^k$. Therefore $v\in H$ and we are done.
\end{proof}
By $\mathcal{H}(\Lambda)$ and $\mathcal{L}(T_\Lambda)$ we respectively denote the lattice of hereditary saturated subsets of $\Lambda^0$ and the lattice of $\mathbb{Z}^k$-order ideals of $T_\Lambda$. Using Lemma \ref{lem order ideal generated by HS subset} and \ref{lem recovering HS subset}, we obtain the following result. 

\begin{prop}\label{pro the lattice isomorphism}

Define $\rho:\mathcal{H}(\Lambda) \longrightarrow \mathcal{L}(T_\Lambda)$ by $\rho(H):=\langle H\rangle$ and $\eta:\mathcal{L}(T_\Lambda) \longrightarrow \mathcal{H}(\Lambda)$ by $\eta(J):=H_{_J}$. Then $\rho$ and $\eta$ are order preserving mutually inverse maps and hence $\mathcal{H}(\Lambda)$ is lattice isomorphic with $\mathcal{L}(T_\Lambda)$.
\end{prop}

Let $R$ be a unital commutative ring and $\KP_R(\Lambda)$ the Kumjian-Pask algebra of $\Lambda$ over $R$ (see Definition \ref{def KP family}). We denote $\mathcal{L}^{\gb}(\KP_R(\Lambda))$ to be the lattice of all graded basic ideals of $\KP_R(\Lambda)$. By using Proposition \ref{pro the lattice isomorphism} and \cite[Theorem 5.1]{Pino} together, we can say that the three lattices $\mathcal{H}(\Lambda)$, $\mathcal{L}(T_\Lambda)$ and $\mathcal{L}^{\gb}(\KP_R(\Lambda))$ are mutually lattice isomorphic. The lattice isomorphism between $\mathcal{L}(T_\Lambda)$ and $\mathcal{L}^{\gb}(\KP_R(\Lambda))$ is given by $J\longmapsto I(H_{_J})$, where $I(H_{_J})$ is the ideal of $\KP_R(\Lambda)$ generated by the hereditary saturated subset $H_{_J}$. 

As an immediate application of these lattice isomorphisms, we obtain the following proposition which is an analogue of \cite[Lemma 5.5]{Hazrat} in the level of higher-rank graphs.

\begin{prop}\label{pro characterizing prime order ideals}
Let $\Lambda$ be a row-finite $k$-graph without sources. Let $J$ be any $\mathbb{Z}^k$-order ideal of $T_\Lambda$. Then $J$ is prime if and only if for any $x,y\in T_\Lambda\setminus J$, there exist $z\in T_\Lambda\setminus J$ and $n,m\in \mathbb{Z}^k$ such that $^{n}x$, $^{m}y\ge z$.   
\end{prop}
\begin{proof}
The sufficiency part follows mutatis-mutandis the proof of \cite[Lemma 5.5]{Hazrat}. For the other way round suppose $J$ is any prime $\mathbb{Z}^k$-order ideal of $T_\Lambda$. We claim that $I(H_{_J})$ is a graded prime ideal of $\KP_\mathsf{k}(\Lambda)$, where $\mathsf{k}$ is any field. Assume that $I_1,I_2$ are graded ideals of $\KP_\mathsf{k}(\Lambda)$ such that $I_1I_2\subseteq I(H_{_J})$. Using the lattice isomorphism between $\mathcal{L}(T_\Lambda)$ and $\mathcal{L}^{\gb}(\KP_\mathsf{k}(\Lambda))$, we can find $\mathbb{Z}^k$-order ideals $J_1,J_2$ of $T_\Lambda$ such that $I_i=I(H_{J_i})$ for $i=1,2$. Since any graded ideal of $\KP_\mathsf{k}(\Lambda)$ is idempotent (see \cite[Theorem 5.1 and Lemma 5.4]{Pino}), so $I_1I_2=I_1\cap I_2$. Then 
\[I(H_{J_1\cap J_2})=I(H_{J_1})\cap I(H_{J_2})=I_1\cap I_2=I_1I_2\subseteq I(H_{_J})\]
which implies $J_1\cap J_2\subseteq J$. Hence $J_1\subseteq J$ or $J_2\subseteq J$ since $J$ is a prime order ideal. So $I_1\subseteq I(H_{_J})$ or $I_2\subseteq I(H_{_J})$ and our claim is established. It follows that $\Lambda^0\setminus H_{_J}$ is a maximal tail (see \cite[Definition 3.10]{Kang-Pask}) by \cite[Proposition 4.1]{Larki}. Let $x=\displaystyle{\sum_{i=1}^{p}} v_i(n_i)$, $y=\displaystyle{\sum_{j=1}^{q}} w_j(m_j) \in T_\Lambda\setminus J$. Since $J$ is an order-ideal and $J=\langle H_{_J}\rangle$ (by Theorem \ref{pro the lattice isomorphism}), so there must exist indices $i\in \{1,2,...,p\}$ and $j\in \{1,2,...,q\}$ such that $v_i,w_j\in \Lambda^0\setminus H_{_J}$. By the definition of maximal tail, there exists $u\in \Lambda^0\setminus H_{_J}$ such that $v_i\Lambda u,w_j\Lambda u\neq \emptyset$. Choose $\lambda\in v_i\Lambda u$ and $\mu\in w_j\Lambda u$. Then in $T_\Lambda$, $x\ge v_i(n_i)\ge u(d(\lambda)+n_i)$ and $y\ge w_j(m_j)\ge u(d(\mu)+m_j)$. Finally letting $n:=-(d(\lambda)+n_i)$, $m:=-(d(\mu)+m_j)$ we have $^{n}x$, $^{m}y\ge u$. This completes the proof.
\end{proof}

Let $v\in \Lambda^0$ and $\overline{v}$ be the smallest hereditary and saturated subset containing $v$. An explicit description of $\overline{v}$ can be found in \cite[Lemma 5.1]{Raeburn}. However for the sake of completion, we exhibit a way of constructing $\overline{v}$. 

For any subset $X\subseteq \Lambda^0$ and $n\in \mathbb{N}^k$, we denote
\begin{center}
    $X^n:=\{u\in \Lambda^0~|~s(u\Lambda^n)\subseteq X\}$
\end{center}
and set $\mathcal{S}(X):=\displaystyle{\bigcup_{n\in \mathbb{N}^k}} X^n$. Now for $v\in \Lambda^0$, let $X_0:=s(v\Lambda)$ and $X_{i+1}:=\mathcal{S}(X_i)$ for all $i\in \mathbb{N}$. 
\begin{lem}\label{lem the construction of HS subset containg v}
With notations as above, $\overline{v}=\displaystyle{\bigcup_{i\in \mathbb{N}}} X_i$ for any $v\in \Lambda^0$.
\end{lem}
\begin{proof}
It is easy to observe that any hereditary saturated subset containing $v$ indeed contains $\displaystyle{\bigcup_{i\in \mathbb{N}}} X_i$. So if we can show that $\displaystyle{\bigcup_{i\in \mathbb{N}}} X_i$ is hereditary and saturated then we are done. First we show that $X_i$ is hereditary for each $i\in \mathbb{N}$ and this will establish that their union is hereditary. For this we use induction. Clearly $X_0=s(v\Lambda)$ is hereditary. Suppose $X_i$ is hereditary for some $i\ge 0$ and $r(\lambda)\in X_{i+1}$ for some $\lambda\in \Lambda$. Then $r(\lambda)\in X_i^n$ for some $n\in \mathbb{N}^k$. Then $s(\alpha)\in X_i$ for all $\alpha\in r(\lambda)\Lambda^n$. Let $m:=(n\vee d(\lambda))-d(\lambda)$. We claim that $s(\mu)\in X_i$ for all $\mu\in s(\lambda)\Lambda^m$. Let $\mu\in s(\lambda)\Lambda^m$. Then $\lambda\mu\in r(\lambda)\Lambda^{n\vee d(\lambda)}$. By factorization property, there exists unique $\nu_1,\nu_2\in \Lambda$ such that $\lambda\mu=\nu_1\nu_2$ with $d(\nu_1)=n$ and $d(\nu_2)=(n\vee d(\lambda))-n$. Then $\nu_1\in r(\lambda)\Lambda^n$ and so $r(\nu_2)=s(\nu_1)\in X_i$ which implies $s(\mu)=s(\nu_2)\in X_i$ since $X_i$ is hereditary. This shows that $s(\lambda)\in X_i^m\subseteq X_{i+1}$. Therefore $X_{i+1}$ is hereditary. Next we show that $\displaystyle{\bigcup_{i\in \mathbb{N}}} X_i$ is saturated. Let $u\in \Lambda^0$ be such that $s(u\Lambda^m))\subseteq \displaystyle{\bigcup_{i\in \mathbb{N}}} X_i$ for some $m\in \mathbb{N}^k$. Since $\Lambda$ is row-finite and $X_i\subseteq X_{i+1}$ so there exists $t\in \mathbb{N}$ such that $s(\alpha)\in X_{t+1}$ for all $\alpha\in u\Lambda^m$ (note that if $s(u\Lambda^m)\subseteq X_0$ then $u\in X_0^m\subseteq X_1\subseteq \displaystyle{\bigcup_{i\in \mathbb{N}}} X_i$ and so we only consider the case where $s(u\Lambda^m)\subseteq X_j$ for some $j\ge 1$). Now for each $\alpha\in u\Lambda^m$, there is $n_\alpha\in \mathbb{N}^k$ such that $s(\alpha)\in X_t^{n_\alpha}$. Set $n:=m+\displaystyle{\bigvee_{\alpha\in u\Lambda^m}} n_\alpha$. We will show that $u\in X_t^n$. So take any $\lambda\in u\Lambda^n$. Let $\mu:=\lambda(0,m)$. Since $\mu\in u\Lambda^m$ we have $s(\mu)\in X_t^{n_\mu}$. Applying unique factorization again, we obtain unique $\gamma_1,\gamma_2\in \Lambda$ such that $\lambda(m,n)=\gamma_1\gamma_2$ with $d(\gamma_1)=n_\mu$ and $d(\gamma_2)=n-m-n_\mu$. Then $\gamma_1\in s(\mu)\Lambda^{n_\mu}$ which implies $r(\gamma_2)=s(\gamma_1)\in X_t$ and consequently $s(\gamma_2)\in X_t$ since $X_t$ is hereditary. Now $s(\lambda)=s(\lambda(m,n))=s(\gamma_2)\in X_t$. Hence we have $u\in X_t^n\subseteq X_{t+1}\subseteq \displaystyle{\bigcup_{i\in \mathbb{N}}} X_i$ and so $\displaystyle{\bigcup_{i\in \mathbb{N}}}X_i$ is saturated. 
\end{proof}

Now we are ready to characterize cofinality of $\Lambda$ using the property of the monoid $T_\Lambda$.
\begin{prop}\label{pro sufficient condition for cofinality}
Let $\Lambda$ be a row-finite $k$-graph without sources. Then $\Lambda$ is cofinal if and only if $T_\Lambda$ is a simple $\mathbb{Z}^k$-monoid.    
\end{prop}
\begin{proof}
The necessity is already proved in Proposition \ref{pro necessary condition for cofinality}. For the sufficiency assume that $T_\Lambda$ is a simple $\mathbb{Z}^k$-monoid. Let $x\in \Lambda^\infty$ be any infinite path and $v\in \Lambda^0$. Consider the nonzero $\mathbb{Z}^k$-order ideal $\langle v\rangle$. Then $\langle v\rangle=T_\Lambda$ and so $\overline{v}=H_{\langle \overline{v}\rangle}=H_{\langle v\rangle}=\Lambda^0$. Therefore $x(0)\in \overline{v}$. So $x(0)\in X_i$ for some $i\in \mathbb{N}$ by Lemma \ref{lem the construction of HS subset containg v}. If $x(0)\in X_0=s(v\Lambda)$ then clearly $v\Lambda x(0)\neq \emptyset$. Assume that $x(0)\in X_{j+1}$ for some $j\ge 0$. Then $x(0)\in \mathcal{S}(X_j)$. So there exists $n_1\in \mathbb{N}^k$ such that $x(0)\in X_j^{n_1}$. Hence $x(n_1)=s(x(0,n_1))\in X_j=\mathcal{S}(X_{j-1})$. Thus there exists $n_2\in \mathbb{N}^k$ such that $x(n_1)\in X_{j-1}^{n_2}$ which implies $x(n_1+n_2)=s(x(n_1,n_1+n_2))\in X_{j-1}$. Since $x$ is an infinite path we can continue this process and after $j$-steps we can find $n_j\in \mathbb{N}^k$ such that
\[x(n_1+n_2+...+n_j)=s(x(n_1+n_2+...+n_{j-1},n_1+n_2+...+n_j))\in X_1=\mathcal{S}(X_0).\] 
Now there exists $m\in \mathbb{N}^k$ such that $x(n_1+n_2+...+n_j+m)\in X_0$. Let $n=n_1+n_2+...+n_j+m$. Then we have $x(n)\in s(v\Lambda)$ and consequently $v\Lambda x(n)\neq \emptyset$. 
\end{proof} 

The following theorem gives a talented monoid criterion for graded basic ideal simplicity of $\KP_R(\Lambda)$. The implications $(iii)\Longleftrightarrow (iv)$ and $(i)\Longleftrightarrow (v)$ follow from \cite[Theorems 5.1, 5.13]{Pino}. Although $(i)\Longrightarrow (iii)$ is proved in \cite[Lemma 5.12]{Pino}, we here obtain an equivalence using the monoid $T_\Lambda$ as a bridge.

\begin{thm}\label{th graded basic simplicity of KPA}
The following statements are equivalent.

$(i)$ $\Lambda$ is cofinal.

$(ii)$ $T_\Lambda$ is a simple $\mathbb{Z}^k$-monoid.

$(iii)$ $\emptyset$ and $\Lambda^0$ are the only hereditary and saturated subsets of $\Lambda^0$.

$(iv)$ $\KP_R(\Lambda)$ is graded basic ideal simple.

$(v)$ The only basic ideal of $\KP_R(\Lambda)$ containing $p_v$ for some $v\in \Lambda^0$ is $\KP_R(\Lambda)$ itself.
\end{thm}

\begin{cor}\label{cor basically simple}
Let $\Lambda$ be a row-finite $k$-graph without sources and $\mathsf{k}$ any field. Then the Kumjian-Pask algebra $\KP_\mathsf{k}(\Lambda)$ is graded simple if and only if $T_\Lambda$ is simple as a $\mathbb{Z}^k$-monoid.
\end{cor}

As the structure of any $\mathbb{Z}^k$-order ideal of $T_\Lambda$ is now clear, we next focus on the quotient of $T_\Lambda$ by any such ideal.

Let $H\in \mathcal{H}(\Lambda)$. Define $\rho:T_\Lambda\longrightarrow T_{\Lambda/H}$ on generators as follows:
$$v(n)\longmapsto
	\left\{
	\begin{array}{ll}
		v(n)  & \mbox{if } v\notin H\\
		0  & \mbox{otherwise } 
	\end{array}
	\right.$$ 
and then extend linearly. Since $H$ is hereditary and saturated one can check that $\rho$ is a well-defined surjective homomorphism of $\mathbb{Z}$-monoids. 
\begin{lem}\label{lem quotient k-graph isomorphism}
Let $H$ be any hereditary saturated subset of $\Lambda^0$. Suppose $\equiv_{H}$ is the congruence on $T_\Lambda$ defined by $a\equiv_{H} b$ if and only if $a+x=b+y$ for some $x,y\in \langle H\rangle$. Then $T_\Lambda/\equiv_{H}~\cong T_{\Lambda/H}$.    
\end{lem}

\begin{proof}
It suffices to show that the kernel congruence $Ker(\rho)=~\equiv_H$. If $a\equiv_H b$ then it is easy to show that $\rho(a)=\rho(b)$ and hence $(a,b)\in Ker(\rho)$. Let $(a,b)\in Ker(\rho)$. Let $a=a'+h_1$ and $b=b'+h_2$ be any two representations of $a,b\in T_\Lambda$ such that the vertices appearing in $h_1,h_2$ are in $H$ and the vertices in $a',b'$ are outside $H$. Then in $T_{\Lambda/H}$, $a'=b'$. Passing this equality to $M_{\overline{\Lambda/H}}$ via the natural isomorphism, we can find $c\in \mathbb{F}_{{\overline{\Lambda/H}}}$ such that $a'\longrightarrow c$ and $b'\longrightarrow c$ (note that we are making some abuse of notation for simplicity). Clearly $\supp(a')$, $\supp(b')$ and $\supp(c)\subseteq (\Lambda^0\setminus H)\times \mathbb{Z}^k$. Let us analyze $a'\longrightarrow c$. It implies that there is a chain of arrows:
\begin{center}
    $a'\longrightarrow_{i_1}\alpha_1\longrightarrow_{i_2}\alpha_2\longrightarrow_{i_3}\cdot\cdot\cdot\longrightarrow_{i_l}\alpha_l=c$ in $\mathbb{F}_{\overline{\Lambda/H}}$. 
\end{center}
Then $a'=a''+a_0$ for some $a_0\in (\Lambda^0\setminus H)\times \mathbb{Z}^k$ and $\alpha_1=a''+\mathfrak{s}^{i_1}(a_0)$ where $\mathfrak{s}^{i_1}(a_0)=\displaystyle{\sum_{\lambda\in a_0\overline{\Lambda/H}^{e_{i_1}}}} \overline{s}(\lambda)$. Therefore in $\mathbb{F}_{\overline{\Lambda}}$, $a'\longrightarrow_{i_1}\alpha_1+h_{\alpha_1}$ for some $h_{\alpha_1}$ with $\supp(h_{\alpha_1})\subseteq H\times \mathbb{Z}^k$. Inductively we can show that $a'\longrightarrow c+h_a$ for some $h_a$ with $\supp(h_a)\subseteq H\times \mathbb{Z}^k$. Similarly $b'\longrightarrow c+h_b$ in $\mathbb{F}_{\overline{\Lambda}}$ for some $h_b$ with $\supp(h_b)\subseteq H\times \mathbb{Z}^k$. Then $a'+h_b,b'+h_a\longrightarrow c+h_a+h_b$ and so $a'+h_b=b'+h_a$ in $M_{\overline{\Lambda}}$. Now returning to $T_\Lambda$ via the inverse of the natural isomorphism, we have $a'+h_b=b'+h_a$ in $T_\Lambda$ (here also we make abuse of notation). Hence $a+h_2+h_b=b+h_1+h_a$. Finally since $h_a+h_1,h_b+h_2\in \langle H\rangle$, so $a\equiv_H b$. 
\end{proof}

\begin{thm}\label{th lattice isomorphism between basic ideals and HS subsets}
Let $\Lambda$ be a row-finite $k$-graph without sources and $R$ any unital commutative ring. Suppose $\mathbb{Z}^k$ acts freely on every quotient of $T_\Lambda$. Then $\mathcal{L}(T_\Lambda)$ is lattice isomorphic with $\mathcal{L}^b(\KP_R(\Lambda))$, the lattice of all basic ideals of $\KP_R(\Lambda)$. Consequently, $T_\Lambda$ is $\mathbb{Z}^k$-simple if and only if $\KP_R(\Lambda)$ is basic ideal simple.
\end{thm}
\begin{proof}
Let $H\in \mathcal{H}(\Lambda)$. By the hypothesis $\mathbb{Z}^k$ acts freely on $T_\Lambda/\equiv_H$ which, by Lemma \ref{lem quotient k-graph isomorphism}, is isomorphic to $T_{\Lambda/H}$. Hence $\Lambda/H$ is aperiodic by Proposition \ref{pro sufficient condition of aperiodicity}. Now the conclusion follows by applying \cite[Corollary 5.7]{Pino} together with Proposition \ref{pro the lattice isomorphism}. 
\end{proof}

\section{Atoms and aperiodic elements of $T_\Lambda$}\label{sec aperiodicity and KP algebra}
In this section we first derive a sufficient condition for aperiodicity of $\Lambda$ purely in terms of the talented monoid $T_\Lambda$.

\begin{dfn}\label{def periodic pair}
Let $(\Lambda,d)$ be a finitely aligned $k$-graph. Let $\alpha,\beta$ be distinct paths in $\Lambda$ with $s(\alpha)=s(\beta)$. We say that $(\alpha,\beta)$ is a \textit{periodic pair} in $\Lambda$ if for any $\eta\in \Lambda$ with $r(\eta)=s(\alpha)=s(\beta)$ we have $\Lambda^{\min}(\alpha\eta,\beta\eta)\neq \emptyset$.
\end{dfn}

Note that a $k$-graph $\Lambda$ is aperiodic in the sense of Lewin and Sims (see Definition \ref{def LS convention}) if and only if there is no periodic pair in $\Lambda$. We prove that if there is a periodic pair in a row-finite $k$-graph $\Lambda$ then $T_\Lambda$ has a nonzero periodic element with respect to the action of $\mathbb{Z}^k$. It gives us the following talented monoid criterion of aperiodicity which can be seen as a partial generalization of \cite[Corollary 4.3 $(i)$]{Hazrat}.

\begin{prop}\label{pro sufficient condition of aperiodicity}
Let $\Lambda$ be a row-finite $k$-graph such that $\mathbb{Z}^k$ acts freely on $T_\Lambda$. Then $\Lambda$ is aperiodic.    
\end{prop}
\begin{proof}
We prove the contrapositive statement. Suppose $(\alpha,\beta)$ is a periodic pair in $\Lambda$ and $m:=d(\alpha), n:=d(\beta)$. Then $\Lambda^{\min}(\alpha,\beta)=\Lambda^{\min}(\alpha s(\alpha),\beta s(\beta))\neq \emptyset$. It implies $m\neq n$. Let $v=s(\alpha)=s(\beta)$, $m':=(m\vee n)-m$ and $n':=(m\vee n)-n$. We claim that the three sets $v\Lambda^{m'},v\Lambda^{n'}$ and $\Lambda^{\min}(\alpha,\beta)$ are in bijective correspondences. Let $\mu\in v\Lambda^{m'}$. Since $(\alpha,\beta)$ is a periodic pair, so $\Lambda^{\min}(\alpha\mu,\beta\mu)\neq \emptyset$. Thus there exists $\overline{\mu}\in \Lambda$ such that $d(\overline{\mu})=d(\alpha\mu)\vee d(\beta\mu)$ and $\overline{\mu}(0,m\vee n)=\alpha\mu$ and $\overline{\mu}(0,d(\beta\mu))=\beta\mu$. Now $d(\beta)=n\le m\vee n=d(\alpha\mu)\le d(\overline{\mu})$. So by factorization property, there exist unique $\beta',\nu_{\mu}\in \Lambda$ such that $\alpha\mu=\beta'\nu_{\mu}$, $d(\beta')=n$ and $d(\nu_{\mu})=(m\vee n)-n=n'$. But then $\overline{\mu}(0,n)=\beta'$ which shows that $\beta'=\beta$ and so $\alpha\mu=\beta\nu_{\mu}$. Hence $(\mu,\nu_{\mu})\in \Lambda^{\min}(\alpha,\beta)$. Note that the choice of $\nu_{\mu}$ does not depend on the choice of the minimal common extension $\overline{\mu}$ of $\alpha\mu$ and $\beta\mu$. Thus we have a well-defined map 
\begin{align*}
\phi:v\Lambda^{m'} & \longrightarrow \Lambda^{\min}(\alpha,\beta)\\
\mu & \longmapsto \phi(\mu):=(\mu,\nu_{\mu}).
\end{align*}
Since $\Lambda$ is cancellative, the map $\phi$ is bijective. Using similar construction one can show that the map \[\psi:\Lambda^{\min}(\alpha,\beta)\longrightarrow v\Lambda^{n'};(\gamma,\delta)\longmapsto \delta\] is also bijective. Since $\Lambda$ is row-finite, $|\Lambda^{\min}(\alpha,\beta)|< \infty$. Let $\Lambda^{\min}(\alpha,\beta)=\{(\mu_1,\nu_1),(\mu_2,\nu_2),...,(\mu_r,\nu_r)\}$. Now using the bijection $\phi$ we have,

\begin{center}
$v=\displaystyle{\sum_{\mu\in v\Lambda^{m'}}} s(\mu)(m')=\displaystyle{\sum_{i=1}^{r}} s(\mu_i)(m')=~^{m'}(\displaystyle{\sum_{i=1}^{r}} s(\mu_i))$.    
\end{center}

Similarly using the bijection $\psi$ and defining relations of $T_\Lambda$ we can have $v=~^{n'}(\displaystyle{\sum_{i=1}^{r}} s(\nu_i))$. Now since $s(\mu_i)=s(\nu_i)$, so $^{m'}a= ~^{n'}a$ where $a:=\displaystyle{\sum_{i=1}^{r}} s(\mu_i)\in T_\Lambda$. Thus $^{m'-n'}a=a$ and $m'-n'\in \mathbb{Z}^k\setminus \{0\}$. So we have a nonzero periodic element in $T_\Lambda$. 
\end{proof}
The converse of Proposition \ref{pro sufficient condition of aperiodicity} is not true in general. We may find a row-finite $k$-graph $\Lambda$ which is aperiodic but $T_\Lambda$ contains a periodic element.
\begin{example}\label{ex counter example for aperiodicity}
Let $A_2$ be the arrow graph with two vertices and one edge. Let $\Lambda:=f^*(A_2^*)$, the $2$-graph defined according to Example \ref{ex higher-rank graphs} $(v)$, where $f:\mathbb{N}^2\longrightarrow \mathbb{N}$ is the monoid homomorphism defined by $f((x,y)):=x+y$ for all $(x,y)\in \mathbb{N}^2$. The $1$-skeleton is shown in the following diagram: 

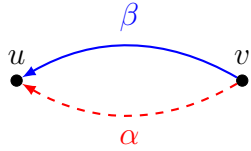
\[
\begin{tikzpicture}[scale=1.5]
\node[inner sep=1.5pt, circle,draw,fill=black,label={$u$}] (A) at (0,0) {};	
\node[inner sep=1.5pt, circle,draw,fill=black,label={$v$}] (B) at (2,0) {};

\path[->, red, dashed, >=latex,thick] (B) edge [bend left] node[below=0.05cm]{$\alpha$} (A);
\path[->,blue, >=latex,thick] (B) edge [bend right] node[above=0.05cm]{$\beta$} (A);
\end{tikzpicture}
\]

Clearly $\Lambda$ is a row-finite $2$-graph. There are no distinct paths with common source $u$. The only pair of distinct paths $\alpha,\beta$ with $s(\alpha)=s(\beta)=v$, $r(\alpha)=r(\beta)$ and $d(\alpha) \wedge d(\beta)=0$ are given as $\alpha=(e,(0,1))$ and $\beta=(e,(1,0))$. If we take $\tau=(v,(0,0))\equiv v$ then it is clear that $\MCE(\alpha\tau,\beta\tau)=\emptyset$. Therefore by \cite[Remark 3.2]{Lewin}, $\Lambda$ is aperiodic.

However note that in $T_\Lambda$,
\begin{center}
    $^{(1,0)}v=v((1,0))=u((0,0))=v((0,1))=~ ^{(0,1)}v$
\end{center}
and so $^{(1,-1)}v=~ ^{(0,0)}v=v$, which shows that $v$ is a nonzero periodic element in $T_\Lambda$.
\end{example}

\begin{rmk}\label{rem about the example}
In \cite[Corollary 4.3 $(i)$]{Hazrat} it was shown that a row-finite graph satisfies Condition $(L)$ if and only if $\mathbb{Z}$ acts freely on the talented monoid $T_E$. This result does not depend on the existence of sink in the graph. In the $k$-graph setting aperiodic paths are evidently the right analogue of cycles with exits (see \cite[Lemma 4.6]{Pino}). From Example \ref{ex counter example for aperiodicity} it may appear that existence of source plays a role for higher-rank graphs. However in the last section we will give another example (see Example \ref{ex simple but not semisimple}) which shows that the converse may not hold even for row-finite $k$-graphs without sources.
\end{rmk}

\begin{dfn}\label{def line point}
(\cite[Definition 3.2]{Brown}) Let $\Lambda$ be any $k$-graph. Then $v\in \Lambda^0$ is called a \textit{line point} if $v\Lambda^\infty=\{x\}$ and $x$ is aperiodic (i.e., if there is exactly one aperiodic infinite path with range $v$). The set of all line points of $\Lambda$ is denoted as $P_l(\Lambda)$.   
\end{dfn}
We obtain the following characterization of line points via $T_\Lambda$.
\begin{thm}\label{th line points characrerization}
Let $\Lambda$ be a row-finite $k$-graph without sources and $v\in \Lambda^0$. Then $v\in P_l(\Lambda)$ (i.e., $v$ is a line point) if and only if $v=v(0)$ is an aperiodic atom in $T_\Lambda$.    
\end{thm}
\begin{proof}
Suppose $v$ is a line point in $\Lambda^0$. Let $v\Lambda^\infty=\{y\}$, where $y$ is an aperiodic infinite path. Suppose $0\neq a\in T_\Lambda$ and $a\le v$. Then $a+b=v$ for some $b\in T_\Lambda$. Passing this equality to $M_{\overline{\Lambda}}$ via the isomorphism $\psi:T_\Lambda \longrightarrow M_{\overline{\Lambda}}$ we get $\psi(a)+\psi(b)=(v,0)$ in $M_{\overline{\Lambda}}$. Then there exists $\gamma\in \mathbb{F}_{_{\overline{\Lambda}}}\setminus \{0\}$ such that $\psi(a)+\psi(b)\longrightarrow \gamma$ and $(v,0)\longrightarrow \gamma$. Since $|v\Lambda^\infty|=1$ so $|v\Lambda^m|=1$ for all $m\in \mathbb{N}^k$. It follows that $\gamma=\overline{s}(\lambda)$ for some $\lambda\in (v,0)\overline{\Lambda}$. Since $a\neq 0$ so $\psi(a)\neq 0$. Thus $\psi(a)\longrightarrow \gamma$ and $\psi(b)=0$. So $\psi(a)=(v,0)$ in $M_{\overline{\Lambda}}$. Then in $T_\Lambda$ we have $a=v$, which proves that $v$ is a minimal element in $T_\Lambda$. Since $T_\Lambda$ is conical and cancellative so $v$ is an atom in $T_\Lambda$. Next we show that $v$ is an aperiodic element. On the contrary suppose there exist distinct $n,m\in \mathbb{Z}^k$ such that $^n v=~^m v$. Note that we can take $n,m$ both from $\mathbb{N}^k$. Now in $\mathbb{F}_{\overline{\Lambda}}$, $(v,n)\longrightarrow \delta$ and $(v,m)\longrightarrow \delta$ for some $\delta \in \mathbb{F}_{\overline{\Lambda}}\setminus \{0\}$. Since $|v\Lambda^N|=1$ for all $N\in \mathbb{N}^k$, $\delta$ is a vertex in $\overline{\Lambda}$. Hence there are paths $(\lambda,n),(\mu,m)\in \overline{\Lambda}$ such that $\overline{r}(\lambda,n)=(v,n)$, $\overline{r}(\mu,m)=(v,m)$ and $\delta=\overline{s}(\lambda,n)=\overline{s}(\mu,m)$. Hence $s(\lambda)=s(\mu)$, $r(\lambda)=r(\mu)=v$ and $n+d(\lambda)=m+d(\mu)$. Since $v\Lambda^\infty=\{y\}$, so $y(0,d(\lambda))=\lambda$, $y(0,d(\mu))=\mu$ and $y(d(\lambda))=s(\lambda)=s(\mu)=y(d(\mu))$. Therefore $\sigma^{d(\lambda)}(y)=\sigma^{d(\mu)}(y)$ with $d(\lambda)\neq d(\mu)$. So $y$ is periodic, a contradiction. Hence $v$ must be aperiodic in $T_\Lambda$. 

Conversely suppose $v$ is an aperiodic atom in $T_\Lambda$. Using the defining relations of $T_\Lambda$, it is easy to show that $|v\Lambda^m|=1$ for all $m\in \mathbb{N}^k$ and hence $|v\Lambda^\infty|=1$. Suppose $v\Lambda^\infty=\{x\}$. If $\sigma^m(x)=\sigma^n(x)$ for some $m,n\in \mathbb{N}^k$ then $x(m)=x(n)$. Let $u=x(m)$. Then in $T_\Lambda$, 
\begin{center}
$u(m)=\displaystyle{\sum_{\alpha\in v\Lambda^m}} s(\alpha)(m)=v=\displaystyle{\sum_{\beta\in v\Lambda^n}}s(\beta)(n)=u(n)$.
\end{center}
Therefore $^n v=~^m v$ which implies $m=n$ since $v$ is aperiodic. Hence $x$ is an aperiodic infinite path and so $v\in P_l(\Lambda)$. 
\end{proof}

\begin{cor}\label{cor partial answer}
Let $\Lambda$ be a row-finite $k$-graph without sources. If $\Lambda$ is aperiodic then $\mathbb{Z}^k$ acts freely on the atoms of $T_\Lambda$. 
\end{cor}
We aim to show that if $\Lambda$ has no source and $T_\Lambda$ is atomic, then the converse of Proposition \ref{pro sufficient condition of aperiodicity} is true. For this we first prove the following interesting fact about $T_\Lambda$.

\begin{lem}\label{lem atomicity implies UFM}
If $T_\Lambda$ is atomic then it is a unique factorization monoid (UFM).
\end{lem}
\begin{proof}
Let $a\in T_\Lambda$ and 
\begin{center}
    $a=\displaystyle{\sum_{i=1}^{N}}u_i(n_i)=\displaystyle{\sum_{j=1}^{M}} v_j(m_j)$
\end{center}
be any two decomposition of $a$ as sum of atoms. Then in $M_{\overline{\Lambda}}$, $\displaystyle{\sum_{i=1}^{N}}(u_i,n_i)=\displaystyle{\sum_{j=1}^{M}} (v_j, m_j)$. By applying Lemma \ref{lem the confluence lemma} and then Lemma \ref{lem the splitting lemma}, we have $\alpha_1,\alpha_2,...,\alpha_N$, $\beta_1,\beta_2,...,\beta_M\in \mathbb{F}_{_{\overline{\Lambda}}}\setminus \{0\}$ such that $\alpha_1+\alpha_2+...+\alpha_n=\beta_1+\beta_2+...+\beta_M$, $(u_i,n_i)\longrightarrow \alpha_i$ for all $i=1,2,...,N$ and $(v_j,m_j)\longrightarrow \beta_j$ for all $j=1,2,...,M$. Since $\mathbb{F}_{_{\overline{\Lambda}}}$ is the free commutative monoid generated by $\overline{\Lambda}^0$ and $u_i,v_j$ are atoms, so $\alpha_i,\beta_j$'s are vertices of $\overline{\Lambda}$, $N=M$ and $\alpha_i=\beta_{\pi(i)}$ for some permutation $\pi\in S_N$. By Lemma \ref{lem the confluence lemma} we have $(u_i,n_i)=(v_{\pi(i)},m_{\pi(i)})$ in $M_{\overline{\Lambda}}$ and consequently $u_i(n_i)=v_{\pi(i)}(m_{\pi(i)})$ in $T_\Lambda$ for all $i=1,2,...,N$. Hence each element of $T_\Lambda$ has a unique factorization into atoms and so $T_\Lambda$ is a UFM.
\end{proof}
\begin{thm}\label{th Equivalence for aperiodicity in UF}
Let $\Lambda$ be a row-finite $k$-graph without sources such that $T_\Lambda$ is atomic (and hence UFM). Then the following assertions are equivalent.

$(i)$ $\Lambda$ is aperiodic.

$(ii)$ Each atom of $T_\Lambda$ is aperiodic.

$(iii)$ $\mathbb{Z}^k$ acts freely on $T_\Lambda$.
\end{thm}
\begin{proof}
The implication $(i)\Longrightarrow (ii)$ follows from Corollary \ref{cor partial answer} and $(iii)\Longrightarrow (i)$ by Proposition \ref{pro sufficient condition of aperiodicity}. We only prove that $(ii)\Longrightarrow (iii)$. Let us assume on the contrary that $a\in T_\Lambda$ is a nonzero periodic element. Write $a=\displaystyle{\sum_{i=1}^{N}} v_i(t_i)$ where $N\in \mathbb{N}$ and each $v_i(t_i)$ is an atom in $T_\Lambda$. Letting $t=\displaystyle{\bigvee_{i=1}^{N}} t_i$, we can express $a$ as $\displaystyle{\sum_{i=1}^{N}} w_i(t)$ where for each $i=1,2,..,N$, $w_i=s(v_i\Lambda^{t-t_i})$. Since $a$ is periodic, there exists $n\in \mathbb{Z}^k\setminus \{0\}$ such that $^n a=a$. Then $^n (\displaystyle{\sum_{i=1}^{N}} w_i)=~^n(^{-t}a)=~^{-t}a=\displaystyle{\sum_{i=1}^{N}} w_i$. Clearly each $w_i$ is an atom in $T_\Lambda$. Since $T_\Lambda$ is a UFM, so there exists $\pi\in S_{N}$ such that $w_i(n)=w_{\pi(i)}$ for all $i=1,2,...,N$. There must exists a positive integer $l$ such that $\pi^l=id$, the identity permutation. Now a simple calculation yields that $w_i(ln)=w_i$. It implies that each $w_i$ is periodic which leads to a contradiction. Therefore $\mathbb{Z}^k$ acts freely on $T_\Lambda$.
\end{proof}
Using the above theorem and Proposition \ref{pro sufficient condition for cofinality} one can obtain the following:
\begin{cor}\label{cor simplicity of KPA}
Let $\Lambda$ be a row-finite $k$-graph without sources and $\mathsf{k}$ any field. Suppose $T_\Lambda$ is atomic. Then the Kumjian-Pask algebra $\KP_\mathsf{k}(\Lambda)$ is simple if and only if $T_\Lambda$ is $\mathbb{Z}^k$-simple and $\mathbb{Z}^k$ acts freely on $T_\Lambda$.
\end{cor}

The following theorem characterizing strongly aperiodic $k$-graphs can be seen as a higher-rank generalization of \cite[Corollary 4.3 $(ii)$]{Hazrat}.

\begin{thm}\label{th characterizing strong aqperiodicity}
Let $\Lambda$ be a row-finite $k$-graph with no sources. If $\mathbb{Z}^k$ acts freely on every quotient of $T_\Lambda$ by a $\mathbb{Z}^k$-order ideal, then $\Lambda$ is strongly aperiodic. The converse holds if $T_\Lambda$ is atomic.   
\end{thm}
\begin{proof}
Let $H$ be any proper hereditary saturated subset of $\Lambda^0$. Then $T_{\Lambda/H}\cong T_\Lambda/\sim_H=T_\Lambda/\langle H\rangle$ by Lemma \ref{lem quotient k-graph isomorphism}. Therefore $\mathbb{Z}^k$ acts freely on $T_{\Lambda/H}$ and consequently $\Lambda/H$ is aperiodic by Proposition \ref{pro sufficient condition of aperiodicity}. So $\Lambda$ is strongly aperiodic.

Now assume that $T_\Lambda$ is atomic and $\Lambda$ is strongly aperiodic. Let $I$ be any $\mathbb{Z}^k$-order ideal of $T_\Lambda$. Since the action of $\mathbb{Z}^k$ on the zero monoid is vacuously free, we can take $I$ to be a proper ideal. By Proposition \ref{pro the lattice isomorphism}, $I=\langle H\rangle$ for some proper hereditary saturated subset $H$. Now $T_\Lambda/I=T_\Lambda/\sim_H\cong T_{\Lambda/H}$ by Lemma \ref{lem quotient k-graph isomorphism} and $\Lambda/H$ is aperiodic due to strong aperiodicity. If we can show that $T_{\Lambda/H}$ is atomic then we are done by Theorem \ref{th Equivalence for aperiodicity in UF}. So assume that $0\neq y\in T_{\Lambda/H}$. Since $\rho:T_\Lambda\longrightarrow T_{\Lambda/H}$ is surjective so $y=\rho(x)$ for some $x\in T_\Lambda\setminus \{0\}$. Now since $T_\Lambda$ is atomic, we can write $x=\displaystyle{\sum_{i=1}^{l}} v_i(n_i)$ where each $v_i(n_i)$ is an atom in $T_\Lambda$. We can partition $\{1,2,...,l\}=A\sqcup B$ such that $v_k\notin H$ for all $k\in A$ and $v_k\in H$ for all $k\in B$. Then $y=\displaystyle{\sum_{k\in A}} v_k(n_k)$. Since $y\neq 0$, $A\neq \emptyset$. For each $k\in A$ and $m\in \mathbb{N}^k$, $v_k\in (\Lambda/H)^0$ and $\emptyset \neq v_k(\Lambda/H)^m
\subseteq v_k\Lambda^m$. Since $|v_k\Lambda^m|=1$, it follows that $v_k(n_k)$ is an atom in $T_{\Lambda/H}$ for all $k\in A$. Thus $T_{\Lambda/H}$ is atomic.
\end{proof}

We now identify the socle of the Kumjian-Pask algebra via the talented monoid $T_\Lambda$ as well as obtain a necessary and sufficient criterion for the socle to be an essential ideal. 

\begin{prop}\label{pro concering socle and semisimplicity}
Let $\Lambda$ be a row-finite $k$-graph without sources and $\mathsf{k}$ any field. Let $\mathcal{A}$ be the $\mathbb{Z}^k$-order ideal of $T_\Lambda$ generated by all the aperiodic atoms in $T_\Lambda$. Then the following hold.

$(i)$ For $v\in \Lambda^0$, $\KP_\mathsf{k}(\Lambda)p_v$ is a minimal left ideal of $\KP_\mathsf{k}(\Lambda)$ if and only if $v$ is an aperiodic atom in $T_\Lambda$.

$(ii)$ $\SC(\KP_\mathsf{k}(\Lambda))=I(H_{_\mathcal{A}})$ where $I(H_{_\mathcal{A}})$ is the ideal of $\KP_\mathsf{k}(\Lambda)$ generated by the hereditary saturated subset $H_{_\mathcal{A}}$.

$(iii)$ $\SC(\KP_\mathsf{k}(\Lambda))$ is an essential ideal of $\KP_\mathsf{k}(\Lambda)$ if and only if for each $v\in \Lambda^0$ there exists $w\in \Lambda^0$ and $m\in \mathbb{N}^k$ such that $v\ge w(m)$ and $w(m)$ is an aperiodic atom in $T_\Lambda$.

\end{prop}
\begin{proof}
Using Theorem \ref{th line points characrerization} together with \cite[Theorems 3.9, 3.10]{Brown}, one can prove $(i)$ and $(ii)$. Now let us prove $(iii)$. Assume that $\SC(\KP_\mathsf{k}(\Lambda))$ is an essential ideal and take $v\in \Lambda^0$. By \cite[Theorem 3.10 (2)]{Brown}, there exists a line point $w\in \Lambda^0$ such that $v\Lambda w\neq \emptyset$. Choose $\lambda\in v\Lambda w$. Then $w(d(\lambda))$ is an aperiodic atom in $T_\Lambda$ by Theorem \ref{th line points characrerization} and we have 
\begin{center}
    $v=v(0)=\displaystyle{\sum_{\alpha\in v\Lambda^{d(\lambda)}}} s(\alpha)(d(\lambda))=w(d(\lambda))+\displaystyle{\sum_{\alpha\in v\Lambda^{d(\lambda)}\setminus \{\lambda\}}} s(\alpha)(d(\lambda))\ge w(d(\lambda))$.
\end{center}
For the converse, in view of \cite[Theorem 3.10 (2)]{Brown}, we just need to show that every vertex in $\Lambda$ connects to a line point. Let $v\in \Lambda^0$. By our hypothesis $v=w(m)+x$ for some aperiodic atom $w(m)\in T_\Lambda$. By using the natural isomorphism $\phi:T_\Lambda\longrightarrow M_{\overline{\Lambda}}$ of Theorem \ref{thm:talisrefine} $(ii)$ we bypass this equality to $M_{\overline{\Lambda}}$ and then using Lemma \ref{lem the confluence lemma}, we have $c\in \mathbb{F}_{\overline{\Lambda}}\setminus \{0\}$ such that $(v,0)\longrightarrow c$, $(w,m)+\phi(x)\longrightarrow c$. Now by Lemma \ref{lem the splitting lemma} there are $c_1,c_2\in \mathbb{F}_{\overline{\Lambda}}\setminus \{0\}$ such that $(w,m)\longrightarrow c_1$, $\phi(x)\longrightarrow c_2$ and $c=c_1+c_2$. Note that $(w,m)\in P_l(\overline{\Lambda})$. This implies $c_1\in P_l(\overline{\Lambda})$ since $P_l(\overline{\Lambda})$ is hereditary. Let $c_1=(u,n)$. Then $u\in P_l(\Lambda)$. As $(v,0)\longrightarrow c=(u,n)+c_2$, it follows that there exists $\lambda\in \Lambda$ such that $r(\lambda)=v$, $s(\lambda)=u$ and $d(\lambda)=n$. Therefore $v$ connects to a line point.
\end{proof}

We say that $v\in \Lambda^0$ is a \textit{leaf} if $|v\Lambda^m|=1$ for all $m\in \mathbb{N}^k$ or equivalently if $|v\Lambda^\infty|=1$. The set of all leaves is denoted by $\mathcal{B}(\Lambda)$. It is easy to see that $\mathcal{B}(\Lambda)$ is a hereditary subset of $\Lambda^0$ and $P_l(\Lambda)\subseteq \mathcal{B}(\Lambda)$.

In the following theorem we characterize semisimple Kumjian-Pask algebra based on the properties of the talented monoid. This also add on some more equivalent conditions to the list provided in \cite[Theorem 4.3]{Brown}. 

\begin{thm}\label{th characterizing semisimplicity}
Let $\Lambda$ be a row-finite $k$-graph without sources and $\mathsf{k}$ any field. Then the following conditions are equivalent.

$(i)$ $\KP_\mathsf{k}(\Lambda)$ is semisimple.

$(ii)$ $T_\Lambda$ is atomic and $\mathbb{Z}^k$ acts freely on $T_\Lambda$.

$(iii)$ $T_\Lambda$ is atomic and each atom of $T_\Lambda$ is aperiodic.

$(iv)$ $\Lambda^0=\overline{\mathcal{B}(\Lambda)}$ and $\Lambda$ is aperiodic as a $k$-graph.
\end{thm}
\begin{proof}
$(i)\Longrightarrow (ii)$ Since $\KP_\mathsf{k}(\Lambda)$ is semisimple, $\overline{P_l(\Lambda)}=\Lambda^0$ and so $T_\Lambda=\langle P_l(\Lambda)\rangle$. Now by using Lemma \ref{lem the splitting lemma} and Lemma \ref{lem the confluence lemma} together with Theorem \ref{th line points characrerization}, it can be shown that each nonzero element $a\in T_\Lambda$ can be written as a sum of aperiodic atoms of $T_\Lambda$. Therefore $T_\Lambda$ is atomic and each atom of $T_\Lambda$ is aperiodic. Now the remaining part follows from Theorem \ref{th Equivalence for aperiodicity in UF}. 

$(ii)\Longrightarrow (iii)$ Obvious.

$(iii)\Longrightarrow (iv)$ Since $T_\Lambda$ is atomic, $T_\Lambda\subseteq \langle \mathcal{B}(\Lambda)\rangle$. Therefore $T_\Lambda=\langle \overline{\mathcal{B}(\Lambda)}\rangle$ and so $\Lambda^0=H_{_{T_\Lambda}}=H_{_{\langle \overline{\mathcal{B}(\Lambda)}\rangle}}=\overline{\mathcal{B}(\Lambda)}$. The remaining part follows from Theorem \ref{th Equivalence for aperiodicity in UF}.

$(iv)\Longrightarrow (i)$ Since $\Lambda$ is aperiodic, each leaf in $\Lambda$ is in fact a line point. Hence $\Lambda^0=\overline{\mathcal{B}(\Lambda)}=\overline{P_l(\Lambda)}$ which implies $\KP_\mathsf{k}(\Lambda)=\SC(\KP_\mathsf{k}(\Lambda))$ by \cite[Corollary 3.11]{Brown}. Therefore $\KP_\mathsf{k}(\Lambda)$ is semisimple.
\end{proof}

\begin{rmk}\label{rem connection with Brown et al result}
Condition $(iv)$ of the above theorem was implicit in the set of equivalent conditions (see \cite[Theorem 4.3]{Brown}) obtained by Brown and an Huef. They proved that for semisimple Kumjian-Pask algebras the $\mathcal{V}$-monoid is atomic. Theorem \ref{th characterizing semisimplicity} shows that in this case $\mathcal{V}^{\gr}(\KP_\mathsf{k}(\Lambda))$ is atomic too. To obtain the above results, besides other things we mainly use the talented monoid realization of line points (Theorem \ref{th line points characrerization}), the key ingredient to describe socle and semisimple Kumjian-Pask algebras in the work of Brown and an Huef. Thus we feel that Proposition \ref{pro concering socle and semisimplicity} and Theorem \ref{th characterizing semisimplicity} have their roots in \cite{Brown}.
\end{rmk}

A complete description of the matrix decomposition of semisimple Kumjian-Pask algebras was proved in \cite{Brown}. It was shown that $\KP_\mathsf{k}(\Lambda)\cong \displaystyle{\bigoplus_{P_l(\Lambda)/\sim_\Lambda}} \mathbb{M}_{\infty}(K)$ where $\mathbb{M}_{\infty}(K)$ is the matrix ring of all $\mathbb{N}\times \mathbb{N}$ matrices with only finitely many nonzero entries and $\sim_\Lambda$ is an equivalence relation on $P_l(\Lambda)$ defined as follows: 

For $v,w\in P_l(\Lambda)$ with $v\Lambda^\infty=\{x\},w\Lambda^\infty=\{y\}$, 
\begin{center}
$v\sim_\Lambda w \Longleftrightarrow x(m)=y(n)$ for some $m,n\in \mathbb{N}^k$.
\end{center}

If $\KP_\mathsf{k}(\Lambda)$ is semisimple then we can provide a necessary and sufficient criterion for $T_\Lambda$ to be finitely generated as a $\mathbb{Z}^k$-monoid. First note that if $P_l(\Lambda)/\sim_\Lambda$ is finite then choosing a representative from each equivalence class we can form a finite generating set for $T_\Lambda$. Conversely suppose $\{a_1,a_2,...,a_m\}$ is a finite generating set for $T_\Lambda$ with $a_i\in \Lambda^0$ for all $i=1,2,...,m$. Since each element of $T_\Lambda$ can be written as a sum of aperiodic atoms, we can assume from the beginning that all the $a_i$'s are in fact aperiodic atoms. Then for any line point $v\in \Lambda^0$, there exists $\alpha\in \{1,2,...,m\}$ and $n\in \mathbb{Z}^k$ such that $v=a_\alpha (n)$ in $T_\Lambda$. This implies that there exists $w\in \Lambda^0$ and $m(>n)\in \mathbb{Z}^k$ such that

\begin{center}
    $(v,0)\longrightarrow (w,m)$ and $(a_\alpha,n)\longrightarrow (w,m)$ in $\mathbb{F}_{\overline{\Lambda}}\setminus \{0\}$
\end{center}
 which in turn implies that there are paths $\lambda,\mu\in \Lambda$ with
 \begin{center}
    $r(\lambda)=v$, $r(\mu)=a_\alpha$, $s(\lambda)=s(\mu)=w$,

    $d(\lambda)=m$ and $d(\mu)=m-n$.
 \end{center}
Suppose $v\Lambda^\infty=\{x\}$ and $a_\alpha\Lambda^\infty=\{y\}$. Then we have
\begin{center}
    $x(m)=s(x(0,m))=s(\lambda)=w=s(\mu)=s(y(0,m-n))=y(m-n)$.
\end{center}
Hence $v\sim_\Lambda a_\alpha$ and it proves that $|P_l(\Lambda)/\sim_\Lambda|\le m$. 

We now obtain the final result of this paper.
 
\begin{thm}\label{th isomorphic TM gives isomorphic KPA}
Let $\Lambda$ and $\Omega$ be row-finite $k$-graphs without sources and $\mathsf{k}$ any field. Let $\phi:T_\Lambda\longrightarrow T_\Omega$ be a $\mathbb{Z}^k$-monoid isomorphism. Then $\KP_\mathsf{k}(\Lambda)$ is semisimple if and only if $\KP_\mathsf{k}(\Omega)$ is semisimple. Moreover in this case if $T_\Lambda$ is finitely generated then $\KP_\mathsf{k}(\Lambda)\cong$ $\KP_\mathsf{k}(\Omega)$.
\end{thm}
\begin{proof}
The first part follows from Theorem \ref{th characterizing semisimplicity}. Now if $T_\Lambda$ is finitely generated then so is $T_\Omega$ and then by the argument used above, both the sets $P_l(\Lambda)/\sim_\Lambda$ and $P_l(\Omega)/\sim_\Omega$ are finite. For the isomorphism now it suffices to show that these sets are in bijection. We define 
\begin{align*}
\tilde{\phi}:P_l(\Lambda)/\sim_\Lambda & \longrightarrow P_l(\Omega)/\sim_\Omega\\
[v] & \longmapsto [(\pi\circ \phi)(v(0))]
\end{align*}
where $\pi$ is the forgetful homomorphism defined in Theorem \ref{thm:talisrefine} (1). We first show that $\tilde{\phi}$ is well-defined. If $v\in P_l(\Lambda)$ then $v(0)$ is an aperiodic atom in $T_\Lambda$ by Theorem \ref{th line points characrerization}. Then $\phi(v(0))$ is an aperiodic atom in $T_\Omega$ and consequently $\pi(\phi(v(0)))$ is in $P_l(\Omega)$. Let $v,w\in P_l(\Lambda)$ be such that $v\sim_\Lambda w$. Then there exist $m,n\in \mathbb{N}^k$ such that $x(m)=y(n)$ where $v\Lambda^\infty=\{x\}$ and $w\Lambda^\infty=\{y\}$. Let $\phi(v)=u_1(t_1)$ and $\phi(w)=u_2(t_2)$. Then $u_1,u_2$ are aperiodic atoms in $T_\Omega$. Suppose $u_1\Omega^\infty=\{z_1\}$ and $u_2\Omega^\infty=\{z_2\}$. Now in $T_\Omega$ we have $z_1(m)(t_1)=~ ^{t_1}z_1(m)=~ ^{t_1-m}u_1=~ ^{-m}\phi(v)=\phi(x(m))=\phi(y(n))=~ ^{-n}\phi(w)=~ ^{t_2-n}u_2=~ ^{t_2}z_2(n)= z_2(n)(t_2)$. Now passing to $M_{\overline{\Omega}}$ via the natural isomorphism and using Lemma \ref{lem the confluence lemma}, we can argue that there exists $(u,p)\in \overline{\Omega}^0$ such that $(z_1(m),t_1),(z_2(n),t_2)\longrightarrow (u,p)$. It follows that there exist paths $\omega_1,\omega_2\in \Omega$ with $r(\omega_1)=z_1(m)$, $r(\omega_2)=z_2(n)$, $s(\omega_1)=s(\omega_2)=u$ and $d(\omega_1)=p-t_1$, $d(\omega_2)=p-t_2$. Then $z_1(m+p-t_1)=s(z_1(0,m+p-t_1))=s(z_1(0,m)z_1(m,m+p-t_1))=s(z_1(m,m+p-t_1))=s(\omega_1)=u=s(\omega_2)=s(z_2(n,n+p-t_2))=s(z_2(0,n+p-t_2))=z_2(n+p-t_2)$. Therefore $u_1\sim_\Omega u_2$ and $\tilde{\phi}$ is well-defined. A similar argument replacing the role of $\phi$ by $\phi^{-1}$, guarantees that $\tilde{\phi}$ is injective as well. Let $[u]\in P_l(\Omega)/\sim_\Omega$. Then there exists an aperiodic atom $v(i)\in T_\Lambda$ such that $\phi(v(i))=u$. Clearly $v\in P_l(\Lambda)$ and $\tilde{\phi}([v])=[\pi(\phi(v))]=[\pi(u(-i))]=[u]$. This proves that $\tilde{\phi}$ is surjective. Thus our desired bijection is established and now the isomorphism $\KP_\mathsf{k}(\Lambda)\cong$ $\KP_\mathsf{k}(\Omega)$ follows from \cite[Theorem 4.10]{Brown}. 
\end{proof}
\begin{rmk}\label{rem about complete invariance}
  If the Kumjian-Pask algebras $\KP_\mathsf{k}(\Lambda)$ and $\KP_\mathsf{k}(\Omega)$ are isomorphic as $\mathbb{Z}^k$-graded algebras then their graded $\mathcal{V}$-monoids are isomorphic as $\mathbb{Z}^k$-monoids. By Proposition \ref{pro realizing as a direct limit} $(i)$, it implies $T_\Lambda$ and $T_\Omega$ are isomophic as $\mathbb{Z}^k$-monoids. In Theorem \ref{th isomorphic TM gives isomorphic KPA}, we show that for semisimple Kumjian-Pask algebras if the talented monoids are $\mathbb{Z}^k$-isomorphic and finitely generated then the algebras are in fact isomorphic. However they may not be isomorphic as $\mathbb{Z}^k$-graded algebras. The reason is that semsimple Kumjian-Pask algebras of row-finite $k$-graphs without sources are non-unital; moreover according to \cite[Remark 4.1]{Brown} these algebras can be realized as Leavitt path algebras of certain directed graphs. These directed graphs fall in the class considered by the first named author and Va\v{s} in \cite{Vas}. It was shown that $\mathbb{Z}[x,x^{-1}]$-module isomorphism between the graded Grothendieck groups is not enough to ensure graded isomorphism between the corresponding Leavitt path algebras; for more details see \cite[Example 5.2 and Corollary 5.6]{Vas}.
\end{rmk}

\section{Examples and Illustration}\label{sec examples}
In this final section, we present some examples to illustrate our main results as well as to justify the necessity of the conditions appeared therein.

\begin{example}\label{ex grid graph}
Recall the infinite grid graph $\Omega_k$ from Example \ref{ex higher-rank graphs} $(iii)$. It is well known that $\Omega_k$ is cofinal and aperiodic. One can also verify these facts by using the talented monoid $T_{\Omega_k}$. For each $v\in \Omega_k^0=\mathbb{N}^k$ and each $m\in \mathbb{N}^k$, $|v\Lambda^m|=1$ (Note that $v\Lambda^m=\{v+m\}$). Hence every vertex is a leaf and so each generator $v(n)$ of $T_{\Omega_k}$ is an atom. Therefore $T_{\Omega_k}$ is an atomic monoid. Now a straightforward argument using the Confluence Lemma \ref{lem the confluence lemma} ensures that each atom of $T_\Lambda$ is aperiodic. Therefore $\Omega_k$ is aperiodic by Theorem \ref{th Equivalence for aperiodicity in UF}. 

Let $J$ be any nonzero $\mathbb{Z}^k$-order ideal of $T_{\Omega_k}$. Choose any $v\in H_{_J}$. Now for any $w\in \Omega_k^0$, we have $w=(v\vee w)((v\vee w)-w)\in \langle v \rangle\subseteq J$. Hence $J=T_{\Omega_k}$ which shows that $\Omega_k$ is cofinal in view of Theorem \ref{pro sufficient condition for cofinality}. The Kumjian-Pask algebra $\KP_\mathsf{k}(\Omega_k)$ is thus both simple and semisimple; in fact $\KP_\mathsf{k}(\Omega_k)\cong \mathbb{M}_{\infty}(K)$. 
\end{example}

\begin{example}\label{ex neither simple nor semisimple}

Consider a $2$-graph $\Lambda$ with the following $1$-skeleton:
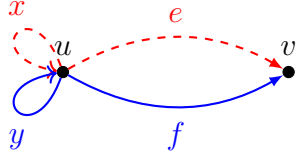
\[
\begin{tikzpicture}[scale=1.5]
\node[inner sep=1.5pt, circle,draw,fill=black,label={$u$}] (A) at (0,0) {};	
\node[inner sep=1.5pt, circle,draw,fill=black,label={$v$}] (B) at (2,0) {};

\path[->, red, dashed, >=latex,thick] (A) edge [bend left] node[above=0.05cm]{$e$} (B);
\path[->,blue, >=latex,thick] (A) edge [bend right] node[below=0.05cm]{$f$} (B);
\path[->, red, dashed,thick, every loop/.style={looseness=35}](A)edge[in=170, out=110, loop] node [above=0.1cm]{$x$}(a) {} (A);
\path[->, blue,thick, every loop/.style={looseness=35}](A)edge[in=190, out=250, loop] node [below=0.1cm]{$y$}(b) {} (A);
\end{tikzpicture}
\]
As usual the dashed red edges have degree $(0,1)$ and the solid blue edges have degree $(1,0)$. The factorization rules are uniquely determined by the skeleton. $\Lambda$ can also be realized as the $2$-graph $f^*(E^*)$ (see Example \ref{ex higher-rank graphs} $(v)$) where $E$ is the following $C_1$-comet:

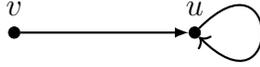
\[
\begin{tikzpicture}[scale=1.2]
\node[inner sep=1.5pt, circle,draw,fill=black,label={$v$}] (A) at (0,0) {};	
\node[inner sep=1.5pt, circle,draw,fill=black,label={$u$}] (B) at (2,0) {};
	
\path[->,black, >=latex,thick] (A) edge [left] node[above=0.05cm]{} (B);
\path[->, black,thick, every loop/.style={looseness=35}](B)edge[in=315, out=45, loop] node [](d) {} (B);
\end{tikzpicture}
\] 
\noindent
Note that in $T_\Lambda$, $v=u((1,0))$. This shows that $T_\Lambda$ is $\mathbb{Z}^2$-simple and hence $\Lambda$ is cofinal. Therefore $\KP_\mathsf{k}(\Lambda)$ is graded simple by Theorem \ref{th graded basic simplicity of KPA}. From the $1$-skeleton we can see that $u\Lambda^{(1,0)}=\{y\}$, $u\Lambda^{(0,1)}=\{x\}$, $v\Lambda^{(1,0)}=\{f\}$ and $v\Lambda^{(0,1)}=\{e\}$. Therefore $u,v$ both are atoms in $T_\Lambda$ and thus $T_\Lambda$ is atomic. However $\mathbb{Z}^2$ does not act freely on $T_\Lambda$ since $^{(1,0)}u=u$. So $\Lambda$ is not aperiodic by Theorem \ref{th Equivalence for aperiodicity in UF} and $\KP_\mathsf{k}(\Lambda)$ is neither simple nor semisimple by applying Corollary \ref{cor simplicity of KPA} and Theorem \ref{th characterizing semisimplicity}. The fact that $\Lambda=f^*(E^*)$ is periodic follows also from \cite[Remark 4.4]{Kumjian-Pask}.
\end{example}

If we determine the $k$-graph monoids for the grid graph $\Omega_2$ and the $2$-graph $\Lambda$ of Example \ref{ex neither simple nor semisimple}, then we can see that $M_{\Omega_2}$ and $M_{\Lambda}$ both are isomorphic to $\mathbb{N}$. However their talented monoids differ. One can check that $T_{\Omega_2}\cong \displaystyle{\bigoplus_{\mathbb{Z}^2}} ~\mathbb{N}$, on the other hand since in $T_\Lambda$, $v=u((1,0))=u((0,1))=u$ and subsequently $^{n}u=u$ for all $n\in \mathbb{Z}^2$, it follows that $T_\Lambda\cong \mathbb{N}$. Therefore talented monoids are more effective than ordinary $k$-graph monoids in distinguishing Kumjian-Pask algebras. \vspace{0.25cm}

In both of the above examples the talented monoids are atomic. Now we give an example of a $k$-graph with non-atomic talented monoid. Moreover this example serves another purpose. It shows that the converse of Proposition \ref{pro sufficient condition of aperiodicity} does not hold in general.

\begin{example}\label{ex simple but not semisimple}
The following is the $1$-skeleton of a \textit{rank-$2$ Bratteli diagram}, first introduced in \cite{Pask} (for more details see \cite[Definition 4.1 and Example 6.7]{Pask}). This $2$-graph is denoted as $\Lambda(2^\infty)$ and it is a rank-$2$ Bratteli diagram of infinite depth. 

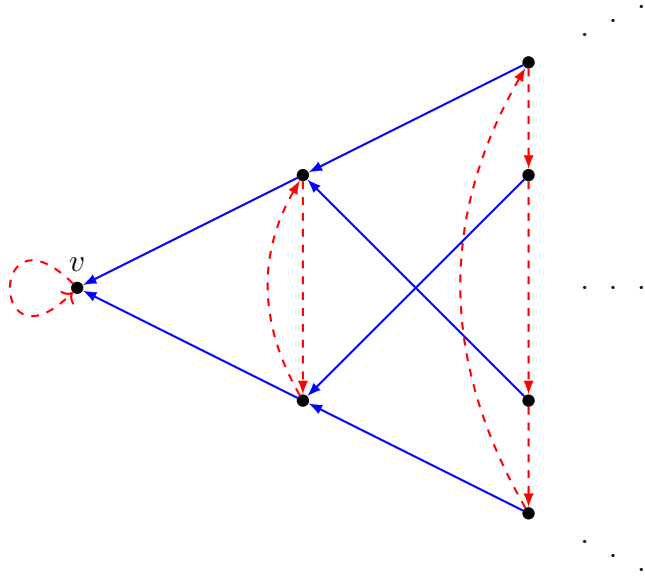
\[
\begin{tikzpicture}[scale=1.5]
\node[inner sep=1.5pt, circle,draw,fill=black,label={$v$}] (A) at (0,0) {};	

\node[inner sep=1.5pt, circle,draw,fill=black] (B1) at (2,1) {};
\node[inner sep=1.5pt, circle,draw,fill=black] (B2) at (2,-1) {};

\node[inner sep=1.5pt, circle,draw,fill=black] (C1) at (4,2) {};
\node[inner sep=1.5pt, circle,draw,fill=black] (C2) at (4,1) {};	
\node[inner sep=1.5pt, circle,draw,fill=black] (C3) at (4,-1) {};
\node[inner sep=1.5pt, circle,draw,fill=black] (C4) at (4,-2) {};	

\node[circle,inner sep=0pt] (u1) at (4.5,2.25) {$.$};
\node[circle,inner sep=0pt] (u2) at (4.75,2.375) {$.$};	
\node[circle,inner sep=0pt] (u3) at (5,2.5) {$.$};

\node[circle,inner sep=0pt] (v1) at (4.5,0) {$.$};
\node[circle,inner sep=0pt] (v2) at (4.75,0) {$.$};	
\node[circle,inner sep=0pt] (v3) at (5,0) {$.$};

\node[circle,inner sep=0pt] (w1) at (4.5,-2.25) {$.$};
\node[circle,inner sep=0pt] (w2) at (4.75,-2.375) {$.$};	
\node[circle,inner sep=0pt] (w3) at (5,-2.5) {$.$};

\path[->, red, dashed, >=latex,thick] (B1) edge [] node[]{} (B2);
\path[->, red, dashed, >=latex,thick] (B2) edge [bend left] node[]{} (B1);
\path[->, red, dashed, >=latex,thick] (C1) edge [] node[]{} (C2);
\path[->, red, dashed, >=latex,thick] (C2) edge [] node[]{} (C3);
\path[->, red, dashed, >=latex,thick] (C3) edge [] node[]{} (C4);
\path[->, red, dashed, >=latex,thick] (C4) edge [bend left] node[]{} (C1);
\path[->, red, dashed,thick, every loop/.style={looseness=35}](A)edge[in=225, out=135, loop] node []{}(e) {} (A);
\path[->,blue, >=latex,thick] (B1) edge [] node[]{} (A);
\path[->,blue, >=latex,thick] (B2) edge [] node[]{} (A);
\path[->,blue, >=latex,thick] (C1) edge [] node[]{} (B1);
\path[->,blue, >=latex,thick] (C3) edge [] node[]{} (B1);
\path[->,blue, >=latex,thick] (C2) edge [] node[]{} (B2);
\path[->,blue, >=latex,thick] (C4) edge [] node[]{} (B2);

\end{tikzpicture}
\]
Clearly $\Lambda(2^\infty)$ is a row-finite $2$-graph without sources. Note that Since $|u\Lambda(2^\infty)^{e_1}|> 1$ for each $u\in (\Lambda(2^\infty))^0$, no vertex is a leaf and so $T_{\Lambda(2^\infty)}$ is not atomic. Hence $\KP_\mathsf{k}(\Lambda(2^\infty))$ is not semisimple by Theorem \ref{th characterizing semisimplicity}. Suppose $I$ is any nonzero $\mathbb{Z}^2$-order ideal of $T_{\Lambda(2^\infty)}$ and $u\in H_{_I}$. If $u$ lies on level $N$, then it is easy to see that for any vertex $w$ on the same level $u=w(le_2)$ for some $0\le l < N$. Thus every vertex of level $N$ is in $H_{_I}$ and consequently $v=\displaystyle{\sum_{\alpha\in v\Lambda(2^\infty)^{Ne_1}}} s(\alpha)(Ne_1)\in I$. This shows that $I=T_{\Lambda(2^\infty)}$ since for each vertex $x$, $v\ge x(m)$ for some $m\in \mathbb{N}^2$. Therefore $T_{\Lambda(2^\infty)}$ is a simple $\mathbb{Z}^2$-monoid and $\Lambda(2^\infty)$ is cofinal. It is also aperiodic by \cite[Theorem 5.1]{Pask} or \cite[Proposition 7.7]{Pino}. The Kumjian-Pask algebra $\KP_\mathsf{k}(\Lambda(2^\infty))$ is thus simple. However $T_{\Lambda(2^\infty)}$ contains periodic elements e.g., $v=v(e_2)$. 
\end{example}
The above example justifies that we can not simply drop the atomicity hypothesis in Theorem \ref{th Equivalence for aperiodicity in UF} and Corollary \ref{cor simplicity of KPA}.

As stated in Theorem \ref{th characterizing semisimplicity}, the two conditions on talented monoids which characterize semisimple Kumjian-Pask algebras are: $(i)$ $T_\Lambda$ is atomic and $(ii)$ every nonzero element of $T_\Lambda$ has trivial isotropy with respect to the natural $\mathbb{Z}^k$-action. We assert the readers that these conditions are independent. Example \ref{ex neither simple nor semisimple} already guarantees the existence of atomic talented monoids on which $\mathbb{Z}^k$ does not act freely and so $(i)$ does not imply $(ii)$. The next example shows that $(ii)$ may not imply $(i)$.

\begin{example}\label{ex non atomic TM with free Z^k action}
Consider the $2$-graph $\Lambda$ whose $1$-skeleton is shown below.
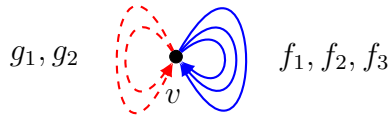
\vspace{-1cm}
\[
\begin{tikzpicture}[scale=0.35]

\node[circle,draw,fill=black,inner sep=0.5pt] (p11) at (1, 1) {$.$} 
edge[-latex, blue, thick,loop, out=40, in=-40, min distance=80, looseness=1.5] (p11)
edge[-latex, blue,thick,loop, out=50, in=-50, min distance=123, looseness=2.2] (p11)
edge[-latex, blue,thick, loop, out=60, in=-60, min distance=190, looseness=2.3] (p11)
edge[-latex, red,thick, loop, dashed, out=125, in=235, min distance=115, looseness=1.8] (p11)
edge[-latex, red, thick, loop, dashed, out=115, in=245, min distance=190, looseness=1.7] (p11);
                                                               
\node at (0.9, -0.5) {$v$};
\node at (-4, 1) {$g_1,g_2$};
\node at (7,1) {$f_1,f_2,f_3$};
\end{tikzpicture}
\]
\vspace{-1cm}

\noindent
We set the factorization rules as: $f_ig_j=g_if_j$ for $i,j=1,2$ and $f_3 g_i=g_if_3$ for $i=1,2$. The basic relations in $T_\Lambda$ are
\begin{equation}\label{eqn basic relations}
    v = v((1,0))+v((1,0))+v((1,0))=3v(e_1) ~\text{and}~ v = v((0,1))+v((0,1))=2v(e_2).
\end{equation}
Note that $|v\Lambda^{e_1}|=3, |v\Lambda^{e_2}|=2$ and so $v$ is not an atom in $T_\Lambda$. Since $v$ is the only vertex of $\Lambda$, it follows that $T_\Lambda$ is not atomic. Now we show that $\mathbb{Z}^2$ acts freely on $T_\Lambda$. On the contrary assume that there exists a periodic element $0\neq x\in T_\Lambda$. Then we can found distinct $m,n\in \mathbb{N}^2$ such that $^m x=~^n x$. Since $\Lambda^0=\{v\}$ and $v$ is not a source, so using the defining relations of $T_\Lambda$, we can write $x=tv(p)$ for some $t\in \mathbb{N}\setminus \{0\}$ and $p\in \mathbb{N}^2$. Then $^m (tv(p))=~^n (tv(p))$ which subsequently implies $tv(m)=tv(n)$. Acting by suitable element of $\mathbb{Z}^2$ if necessary, we can take $m,n$ such that $m\wedge n=0$. There are three cases to consider:

\emph{Case-1:} $m=0,n> 0$. Let $n=n_1e_1+n_2e_2$. By repeated use of the relations in (\ref{eqn basic relations}), we have $v(0)=3^{n_1}2^{n_2}v(n_1e_1+n_2e_2)$. Thus $3^{n_1}2^{n_2}tv(n)=tv(n)$. Now since $T_\Lambda$ is conical and cancellative, this implies $3^{n_1}2^{n_2}=1$ and so $n_1=n_2=0$ which is a contradiction as $n> 0$.

\emph{Case-2:} $m>0,n=0$. This is exactly similar to Case-1.

\emph{Case-3:} $m,n>0$. Without loss of generality assume that $m=ae_1$ and $n=be_2$ where $a,b$ are positive integers. Applying relations of (\ref{eqn basic relations}), we can write $v(ae_1)=3\cdot 2^{b+1}v((a+1,b+1))$ and $v(be_2)=2\cdot 3^{a+1}v((a+1,b+1))$. Thus $t\cdot 3\cdot 2^{b+1}v((a+1,b+1))=t\cdot 2\cdot 3^{a+1} v((a+1,b+1))$ which implies $3^a=2^b$ and hence $a=b=0$, a contradiction.

So there is no periodic element in $T_\Lambda$. By Proposition \ref{pro sufficient condition of aperiodicity}, $\Lambda$ is aperiodic. It is trivially cofinal. Thus the Kumjian-Pask algebra $\KP_\mathsf{k}(\Lambda)$ is simple. However it is not semisimple since $T_\Lambda$ is not atomic. As $\Lambda$ contains cycles, $\KP_\mathsf{k}(\Lambda)$ falls under the class of purely infinite simple algebras in view of \cite[Theorem 6.3]{LPIS}.
\end{example}

{\bf Acknowledgments.} This work was initiated while the authors were attending the FiLPART School, Bohol, Philippines in July, 2024. The authors would like to thank Alfilgen Sebandal for organizing this gathering. The first named author acknowledges Australian Research Council Discovery Project Grant DP230103184. The second and fourth named authors convey their sincere gratitude to (retd.) Prof. M.K. Sen of University of Calcutta, India for his continuous encouragement in their works. The second named author likes to thank University Grants Commission (UGC), India for providing Research Fellowship (ID:
211610022010/ Joint CSIR-UGC NET JUNE 2021).

{\bf Competing Interests.} The authors declare that there are no financial or non-financial competing interests related to this work.

\end{document}